\documentclass[12pt]{amsart}
\usepackage{amssymb,latexsym}
\usepackage[colorlinks=true,linkcolor=blue,citecolor=red,urlcolor=blue]{hyperref}

\title[Legendrian self-shrinkers]{Rigidity theorems for complete Legendrian self-shrinkers}
\author[Chang]{Shu-Cheng Chang}
\address{Shanghai Institute for Mathematics and Interdisciplinary Sciences,\\
Shanghai 200433, China}
\email{scchang@math.ntu.edu.tw}
\author[Qiu]{Hongbing Qiu}
\address{School of Mathematics and Statistics, Wuhan University,\\
Wuhan 430072, China}
\email{hbqiu@whu.edu.cn}
\author[Zhang]{Liuyang Zhang}
\address{School of Mathematical Sciences, Chongqing University of Technology,\\
Chongqing 400054, China}
\email{zhangliuyang@cqut.edu.cn}
\thanks{The first author is partially supported by the Startup Foundation for Advanced
Talents of the Shanghai Institute for Mathematics and Interdisciplinary Sciences
(No. 2302-SRFP-2024-0049); the second author by the National Natural Science Foundation of
China (No. 12471050) and the Hubei Provincial Natural Science Foundation of China
(No. 2024AFB746); the third author by the National Natural Science Foundation of China (No. 12601100) and the Chongqing Natural Science Foundation Innovation and Development Joint Fund (No. CSTB2025NSCQ-LZX0059).}

\DeclareMathOperator{\Div}{div}
\DeclareMathOperator{\Vol}{Vol}
\DeclareMathOperator{\Hess}{Hess}

\theoremstyle{plain}
\newtheorem{thm}{Theorem}[section]
\newtheorem{lem}[thm]{Lemma}
\newtheorem{cor}[thm]{Corollary}
\newtheorem{pro}[thm]{Proposition}
\theoremstyle{definition}
\newtheorem{defi}[thm]{Definition}
\newtheorem{rem}[thm]{Remark}
\newtheorem{exa}[thm]{Example}
\theoremstyle{plain}
\newtheorem*{thmA}{Theorem A}
\newtheorem*{corB}{Corollary B}
\newtheorem*{thmC}{Theorem C}
\newtheorem*{thmD}{Theorem D}

\def\p{\partial}
\def\n{\nabla}
\def\D{\Delta}
\def\d{\delta}
\def\la{\langle}
\def\ra{\rangle}
\def\on{\overline{\nabla}}

\numberwithin{equation}{section}


\begin{document}

\begin{abstract}
We study self-shrinkers of the Legendrian mean curvature flow in the contact Euclidean
space $\mathbb R^{2n+1}$, with the self-shrinker equation defined by the generator of the
contact dilations. Our main result states that a complete properly immersed Legendrian
self-shrinker whose Legendrian angle is bounded from above or from below is a Legendrian
$n$-plane. Since the angle equals one quarter of the height, this gives rigidity in the
half-spaces $\{z\geq c\}$ and $\{z\leq c\}$ and excludes closed
Legendrian self-shrinkers. It also implies that every entire Legendrian self-shrinking
graph has a quadratic potential. We further prove a sharp cylinder theorem without an
angle bound: a properly immersed Legendrian self-shrinker lying on either side of the
critical cylinder $|w|^2=8n$ must lie on its boundary. In dimensions $n\geq2$, the boundary
case is a screw sweep of a compact minimal Legendrian submanifold of a sphere, with a
global Riemannian product decomposition. Explicit properly embedded families establish
sharpness in every dimension. The proofs use the drift equations for the horizontal
potential and the height, together with weighted volume estimates and exhaustion
functions adapted to the properness of the full immersion.
\end{abstract}

\maketitle

\section{Introduction}

The mean curvature flow deforms an immersion $F_t:\Sigma\to(M,\overline g)$ in the direction
of its mean curvature vector, $\p_tF_t=H$; it is the negative gradient flow of the volume
functional, and its stationary points are the minimal submanifolds. For hypersurfaces we
refer to the survey \cite{HP96}. In higher codimension the most studied case is the
Lagrangian one: if the ambient manifold is K\"ahler--Einstein, the mean curvature flow
preserves the class of Lagrangian submanifolds \cite{S96}. Legendrian submanifolds of
Sasakian manifolds are the contact counterparts of Lagrangian submanifolds, but here the
situation is different: the flow $\p_tF=H$ does \emph{not} preserve the Legendrian
condition. Smoczyk \cite{S03} showed that in an $\eta$-Einstein Sasakian manifold the flow
which does preserve it is
\begin{equation}\label{eqn-int-lmcf}
\frac{\p F}{\p t}=H-2\theta\,\xi ,
\end{equation}
the \emph{Legendrian mean curvature flow}, where $\theta$ is the Legendrian angle,
normalised by $H=\Phi\n\theta$, and $\xi$ is the Reeb vector field. Here, as everywhere in
this paper, the mean curvature form of the immersion is assumed to be exact, so that a
globally defined Legendrian angle exists; this is the case treated in
\cite[Definition 3.4]{S03}. For a general Legendrian immersion the flow of
\cite[Definition 3.5]{S03} is written with the Legendrian pseudo-angle and carries an
additional non-local term \cite[Remark 3.6]{S03}. The coefficient $2$ is
not a normalisation. With the convention $d\eta=2\,\overline g(\cdot,\Phi\cdot)$ used
throughout this paper, a normal variation $\Phi W+h\,\xi$ of a Legendrian immersion is
tangent to the space of Legendrian immersions if and only if $W=-\frac12\n h$ (Lemma
\ref{lem-ham}); since $H=\Phi\n\theta$, the vertical correction of $H$ has to be
$-2\theta\xi$. Long-time existence and convergence of \eqref{eqn-int-lmcf} in
$\eta$-Einstein Sasakian manifolds, under a stability condition inspired by the
Thomas--Yau conjecture \cite{TY01}, were established by Chang, Han and Wu \cite{CHW24}.

In general \eqref{eqn-int-lmcf} develops singularities in finite time, and, as for the mean
curvature flow \cite{Hui90}, a Type I singularity is modelled by a self-similarly shrinking
solution \cite{CWZ23}. What ``self-similar'' means has to be read off from the symmetries
of the flow. The flat model is the contact Euclidean space $\mathbb R^{2n+1}$ with the
contact form $\eta=\frac12dz-\frac14\sum_i(y_idx_i-x_idy_i)$, that is, the Heisenberg group
with its standard Sasakian structure, and \eqref{eqn-int-lmcf} is a flow of submanifolds
constrained by $\eta$. The Euclidean dilations do not preserve the contact distribution
$\ker\eta$, hence do not act on the class of Legendrian submanifolds. The dilations which do are
the contact dilations
\begin{equation}\label{eqn-int-dil}
\d_\lambda(x,y,z)=(\lambda x,\lambda y,\lambda^{2}z),\qquad
\d_\lambda^{*}\eta=\lambda^{2}\eta\qquad(\lambda>0),
\end{equation}
in which the contact coordinate $z$ has weight $2$. Their infinitesimal generator is
\begin{equation}\label{eqn-int-X}
X=\sum_{i=1}^n\Bigl(x_i\frac{\p}{\p x_i}+y_i\frac{\p}{\p y_i}\Bigr)+2z\frac{\p}{\p z},
\end{equation}
and a Legendrian immersion $F:\Sigma^n\to\mathbb R^{2n+1}$ is a \emph{Legendrian
self-shrinker} if
\begin{equation}\label{eqn-int-shrinker}
H-2\theta\,\xi=-\frac12X^{\perp} .
\end{equation}
Then $F_t=\d_{\sqrt{-t}}\circ F$, $t<0$, solves \eqref{eqn-int-lmcf} up to tangential
diffeomorphisms. That the right-hand side of \eqref{eqn-int-shrinker} is $X^\perp$, and not
the normal part of the Euclidean position vector, is not a matter of choice; we explain
this in Section \ref{sec-motivation}, in parallel with the discussion of the translating
field in \cite{CQZ26}. Examples of Legendrian self-shrinkers, and classification results
for Legendrian self-shrinking surfaces in $\mathbb R^5$, were obtained in \cite{CWZ24};
translating solitons of \eqref{eqn-int-lmcf}, which are the self-similar solutions with
respect to the left translations of the Heisenberg group, are studied in \cite{CQZ26}.

There are good reasons to study \eqref{eqn-int-lmcf}, and its self-shrinkers in
$\mathbb R^{2n+1}$ in particular, in their own right, and we mention three which bear on
the present paper. The first is that the flow has its own model space. The rescalings
compatible with a Sasakian structure are not the Riemannian homotheties but the
$\mathcal D$-homothetic deformations of Tanno \cite{Tan68}, and a $\mathcal D$-homothety
with parameter $\kappa$ changes a $\Phi$-sectional curvature $-3+s$ into
$-3+\frac{s}{\kappa}$, which tends to $-3$, and not to $0$, as $\kappa\to\infty$.
Accordingly the natural blow-up model of a Sasakian manifold at a point is not a Euclidean
space but the Sasakian space form of\/ $\Phi$-sectional curvature $-3$, which is the contact
Euclidean space $\mathbb R^{2n+1}$; and on it the rescalings are realised by the
automorphisms \eqref{eqn-int-dil}. The self-shrinkers of\/ $\mathbb R^{2n+1}$ are thus
expected to play for
\eqref{eqn-int-lmcf} in an arbitrary $\eta$-Einstein Sasakian manifold the role which the
self-shrinkers of\/ $\mathbb R^{m}$ play for the mean curvature flow in an arbitrary
Riemannian manifold, with the difference that the model is a curved space and a
non-commutative group.

The second reason is that the soliton equation can be expressed as a scalar relation
between the Legendrian angle and the ambient coordinates. Both sides of
\eqref{eqn-int-shrinker} are Legendrian variations, and a
Legendrian variation is determined by its $\xi$-component; hence the vector equation
\eqref{eqn-int-shrinker} is equivalent to the single identity
\begin{equation}\label{eqn-int-scalar}
4\theta=z\big|_{\Sigma^n}
\end{equation}
(Proposition \ref{pro-scalar}), in which no derivative of $\theta$ and no constant of integration
appears. On a Legendrian self-shrinker the Legendrian angle is therefore a coordinate
function: a bound for the angle is a bound for the height of the submanifold, and
conversely. In the same way the Legendrian angle of a translating soliton is an affine
function of $x$ and $y$ \cite{CQZ26}. This is what makes the angle, rather than a quantity
built from the second fundamental form, the natural object to estimate below.

The third reason is topological. The flow \eqref{eqn-int-lmcf} moves a Legendrian
submanifold through Legendrian immersions. On every time interval on which these remain
embeddings, the flow is a Legendrian isotopy and preserves its invariants. Closed Legendrian
submanifolds of\/ $\mathbb R^{2n+1}$ have a rich
isotopy theory, with infinitely many classes which the classical invariants do not
distinguish \cite{EES05}, \cite{Gei08}, and, on the classes which carry a global Legendrian angle,
\eqref{eqn-int-lmcf} is a natural candidate for producing distinguished representatives, in
the spirit of \cite{TY01} and of the convergence theorem of \cite{CHW24}. The singularities are the places where this
can fail, which is one more motivation for understanding their models.

The purpose of this paper is to prove rigidity theorems for complete Legendrian
self-shrinkers. Our main result is the following.

\begin{thmA}
Let $F:\Sigma^n\to\mathbb R^{2n+1}$ be a complete properly immersed Legendrian self-shrinker
whose Legendrian angle $\theta$ is bounded from above or from below. Then $\theta$ is
constant, and $F(\Sigma^n)$ is a Legendrian $n$-plane
\begin{equation*}
\bigl\{(w,c):\ w\in\mathcal P\bigr\}\subset\mathbb R^{2n}\times\mathbb R ,\qquad c=4\theta ,
\end{equation*}
where $\mathcal P$ is an $n$-dimensional linear subspace of\/ $\mathbb R^{2n}$ on which
$\sum_idx_i\wedge dy_i$ vanishes.
\end{thmA}

What distinguishes Theorem A from the rigidity theorems one is used to is the nature of its
hypothesis. By \eqref{eqn-int-scalar} the Legendrian angle of a self-shrinker is the
coordinate function $\frac z4$, so that the hypothesis on the angle is a hypothesis on the
position of the submanifold, and Theorem A is a half-space theorem in the direction of
the Reeb field.

\begin{corB}
Let $F:\Sigma^n\to\mathbb R^{2n+1}$ be a complete properly immersed Legendrian
self-shrinker.
\begin{itemize}
\item[(i)] If\/ $F(\Sigma^n)$ lies in a half-space $\{z\geq c\}$ or $\{z\leq c\}$, then
$F(\Sigma^n)$ is a Legendrian $n$-plane as in Theorem A.
\item[(ii)] $\Sigma^n$ is not compact: there are no closed Legendrian self-shrinkers in
$\mathbb R^{2n+1}$.
\end{itemize}
\end{corB}

Statements of this form are not available for Lagrangian self-shrinkers in $\mathbb C^n$.
There the self-shrinker equation does not contain the Lagrangian angle, which in general is
defined only modulo $2\pi$, and in particular compactness does not imply rigidity: the
circle of radius $\sqrt2$, the Clifford torus \cite{CL14} and the examples of Anciaux
\cite{Anc06} are closed Lagrangian self-shrinkers. A sufficient hypothesis in the
Lagrangian case is that the angle be single-valued, which is the vanishing of the Maslov
class, and bounded; both are conditions on the Gauss map rather than on the position of the
submanifold. In our definition of a Legendrian self-shrinker the angle is single-valued,
and \eqref{eqn-int-scalar} makes a bound on it from one side equivalent to the geometric
condition $z\geq c$ or $z\leq c$; we discuss this in Remark \ref{rem-lagr}. Corollary B (i)
may be compared with the half-space theorem for properly immersed Euclidean self-shrinking
hypersurfaces \cite[Theorem 1.1]{CE16}.

The second application concerns entire graphs. For the mean curvature flow in Euclidean
space the rigidity of entire graphical self-shrinkers is well understood: entire graphical
self-shrinking hypersurfaces are hyperplanes (Ecker--Huisken \cite{EH89} under a polynomial
growth condition, and L.\ Wang \cite{Wan11} in general), and for entire Lagrangian graphs
$\{(x,Du(x))\}$ in $\mathbb C^n$ Chau, Chen and Yuan \cite[Theorem 1.1]{CCY12} proved that the potential
$u$ is a quadratic polynomial by constructing a barrier for the Lagrangian angle; see also
Huang and Wang \cite{HW11} for earlier results, and Ding and Xin \cite{DX14} for another
proof by an integral method. In the contact Euclidean space an entire Legendrian graph over
the $x$-plane is the $1$-jet graph of a function (Lemma \ref{lem-graph}),
\begin{equation}\label{eqn-int-graph}
F_u(x)=\Bigl(x,\ Du(x),\ u(x)-\frac12\,x\cdot Du(x)\Bigr),\qquad x\in\mathbb R^n ,
\end{equation}
which is the global form, in the flat model, of the Legendrian neighbourhood theorem
\cite{Gei08}. Its Legendrian angle is $\theta=-\sum_i\arctan\lambda_i$, where
$\lambda_1,\dots,\lambda_n$ are the eigenvalues of $D^2u$, and by
\eqref{eqn-int-scalar} it is a Legendrian self-shrinker if and only if
\begin{equation}\label{eqn-int-pde}
\sum_{i=1}^{n}\arctan\lambda_i(x)=\frac14\Bigl(\frac12\,x\cdot Du(x)-u(x)\Bigr),
\qquad x\in\mathbb R^n .
\end{equation}
An entire graph is complete and properly embedded, and its Legendrian angle is bounded by
$\frac{n\pi}2$ whatever $u$ is. Theorem A therefore applies without any assumption.

\begin{thmC}
If $u$ is an entire smooth solution of \eqref{eqn-int-pde} in $\mathbb R^n$, then $u$ is
the quadratic polynomial
\begin{equation*}
u(x)=u(0)+\frac12\la D^2u(0)\,x,x\ra ,
\end{equation*}
so that $F_u(\mathbb R^n)$ is the Legendrian $n$-plane $\{(x,D^2u(0)x,u(0))\}$, whose
height is 
\begin{equation*}
u(0)=-4\sum_i\arctan\lambda_i\bigl(D^2u(0)\bigr) .
\end{equation*}

\end{thmC}

The height condition in Theorem A leaves open the geometry of proper immersions with
Legendrian angle unbounded in both directions. Our second rigidity result concerns their horizontal
position. Writing $w=x+\sqrt{-1}\,y$, set
\begin{equation*}
\mathcal C_n:=\{(w,z)\in\mathbb C^n\times\mathbb R:\ |w|^2=8n\} .
\end{equation*}

\begin{thmD}
Let $F:\Sigma^n\to\mathbb R^{2n+1}$ be a properly immersed Legendrian self-shrinker.
If either $|w|^2\leq8n$ everywhere or $|w|^2\geq8n$ everywhere, then
$F(\Sigma^n)\subset\mathcal C_n$. Equivalently, either $f\equiv n/2$ or
\begin{equation*}
\inf_{\Sigma^n}f<\frac n2<\sup_{\Sigma^n}f,\qquad f=\frac1{16}|w|^2 .
\end{equation*}
The threshold is sharp in every dimension, even among properly embedded self-shrinkers.
\end{thmD}

This result belongs to the theory of restrictions on the position of self-shrinkers,
including the halfspace and cylinder theorems of Cavalcante--Espinar \cite{CE16}, their
higher-codimension counterparts \cite{DN24}, and the potential-theoretic approach of
Pigola--Rimoldi \cite{PR14} and Impera--Pigola--Rimoldi \cite{IPR21}. In the metric
$\frac14\sum_i(dx_i^2+dy_i^2)$ of the base $\mathbb C^n$ (Proposition \ref{pro-heis} (iv))
the cylinder $\mathcal C_n$ is the product of the sphere of radius $\sqrt{2n}$ with the
$z$-axis, so that the threshold $8n$ is the contact counterpart of the critical radius
$\sqrt{2n}$ of the Euclidean theory. The cylinder theorem of \cite{CE16} states that a
self-shrinker properly immersed in a closed solid cylinder
$\overline{B^{k+1}(R)}\times\mathbb R^{m-k}\subset\mathbb R^{m+1}$ with $R\leq\sqrt{2k}$ is
the cylinder $S^k(\sqrt{2k})\times\mathbb R^{m-k}$; Theorem D is a statement about both
sides of\/ $\mathcal C_n$, and its cylinder is noncompact in the direction of the height,
which a bound on $f$ does not control. We use
the equation $\mathcal Lz=0$ to construct exhaustion functions which account for this
direction. No bound on the Legendrian angle or assumption of finite weighted volume is
imposed. Sharpness means that the threshold $8n$ can be neither increased on the inner
side nor decreased on the outer side: Propositions \ref{pro-sharp-curve} and
\ref{pro-sharp-radial} give, in every dimension, properly embedded Legendrian
self-shrinkers not contained in $\mathcal C_n$ along which $|w|^2$ stays arbitrarily
close to $8n$. The boundary case also has a global description: for $n\geq2$, the
zero-height slice is a compact minimal Legendrian submanifold of the sphere of radius
$\sqrt{8n}$, and the whole immersion is obtained by its screw motion
$(w,z)\mapsto(e^{-\sqrt{-1}s}w,z+4ns)$ (Theorem \ref{thm-cylinder-structure}); Corollary
\ref{cor-lowdim} classifies the boundary case for $n=1,2$.

Theorem A is Theorem \ref{thm-general} below, Corollary B is Corollary \ref{cor-slab},
Theorem C is Theorem \ref{thm-rigid}, and Theorem D is Theorem \ref{thm-cylinder},
with sharpness proved in Section \ref{subsec-cylinder-examples}. A cut-off argument in the
spirit of Ding and Xin \cite{DX14} proves Proposition \ref{pro-fproper}, which assumes only
that the function $f$ of \eqref{eqn-f} is proper. Theorem A states the hypotheses directly in terms of the
immersion and its height (Remark \ref{rem-fproper}). The hypothesis on the Legendrian angle in Theorem A
cannot be dropped in any dimension: Example \ref{exa-cyl} gives a complete, properly
embedded Legendrian self-shrinker in $\mathbb R^{2n+1}$ which is not a Legendrian
$n$-plane, and its Legendrian angle is unbounded; for
$n=1$ the Legendrian Abresch--Langer curves of \cite{CWZ24} are further such examples
(Example \ref{exa-AL}).

We prove Theorem A by a minimum principle for the height, followed by the identification
of the limiting object: if $z\geq c$, the nonnegative function $z-c$ is annihilated by the
drift Laplacian, and a limit of minimum points of $z-c+\varepsilon f$ as
$\varepsilon\downarrow0$ is a minimum point of $z$. The integral method of Ding and Xin
\cite{DX14}, a weighted volume estimate, a bounded function annihilated by the drift
Laplacian, and a cut-off argument, gives Proposition \ref{pro-fproper} and, when the angle is
bounded, a second proof that it is constant (Remark \ref{rem-twosided}). In the Sasakian
geometry of\/ $\mathbb R^{2n+1}$ these steps take a form of their own, and four points
deserve to be mentioned. First, the role of the position
vector is played by the field $X$ of \eqref{eqn-int-X}, which is not a homothetic vector
field: $\d_\lambda$ is a $\mathcal D$-homothety and not a homothety of\/ $\overline g$, and
along a Legendrian submanifold one has $\on_YX=Y-z\,\Phi Y+\frac12Y(z)\,\xi$ in place of
$\on_YX=Y$ (Lemma \ref{lem-DX}). Secondly, as a consequence, the function which satisfies the
structural identities $\n f=\frac12X^{T}$ and $\D f=\frac n2+\frac12\la X,H\ra$ is
\begin{equation*}
f:=\frac14\bigl|X^{\mathcal H}\bigr|^{2}\Big|_{\Sigma^n},\qquad X^{\mathcal H}:=X-\eta(X)\xi ,
\end{equation*}
the squared length of the horizontal part of $X$, and not $\frac14|X|^2=f+\frac14z^2$;
we call $f$ the horizontal potential. With the weight $e^{-f}$ every complete noncompact
Legendrian self-shrinker on which $f$ is proper
has finite weighted volume and Euclidean volume growth,
\begin{equation*}
\int_{\Sigma^n}e^{-f}d\mu<+\infty,\qquad
\Vol\bigl(\{2\sqrt f\leq r\}\bigr)\leq C_1r^{n}\quad(r\geq1),
\end{equation*}
without any assumption on the Legendrian angle (Theorem \ref{thm-vol}); this follows from
a theorem of Cheng and Zhou \cite{CZ13}, see also Ding and Xin \cite{DX13}, and for the
general theory of self-shrinkers we refer to Colding and Minicozzi \cite{CM12}. Thirdly,
the vertical component of $X$ is controlled by \eqref{eqn-int-scalar}: $z=4\theta$ on a
self-shrinker, and the identity $\D z=2\la X,\n\theta\ra$, valid on every Legendrian
submanifold (Lemma \ref{lem-Dz}), becomes $\mathcal L\theta=0$ for the drift Laplacian
\begin{equation*}
\mathcal L:=\D-\frac12\la X,\n(\cdot)\ra=e^{f}\Div\bigl(e^{-f}\n(\cdot)\bigr)
\end{equation*}
of \cite{CM12}. Fourthly, once $\theta$ is constant the submanifold is minimal and
$\frac12X^{T}$ is the gradient of a function with $\Hess f=\frac12g$; the conclusion is then
drawn from the curvature of the ambient space, which vanishes on the tangent planes of a
Legendrian submanifold although its $\Phi$-sectional curvature is $-3$, and from the
explicit form of the geodesics tangent to the contact distribution.

The paper is organised as follows. In Section \ref{sec-prelim} we describe the contact
Euclidean space as the Heisenberg group with a left-invariant Sasakian structure, and
recall the Legendrian angle, the Legendrian variations and the description of entire
Legendrian graphs. Section \ref{sec-motivation} explains why the vector field in the
self-shrinker equation is the generator of the contact dilations, defines Legendrian
self-shrinkers and reduces \eqref{eqn-int-shrinker} to \eqref{eqn-int-scalar} and, for
graphs, to \eqref{eqn-int-pde}. Section \ref{sec-vol} collects the identities for $X$,
$f$ and $z$ and proves the weighted volume estimate. Section \ref{sec-rigid} proves
Theorem A and its half-space and graphical consequences. Section \ref{sec-cylinder} proves
Theorem D, describes its equality case, and establishes sharpness.

\vskip24pt

\section{Preliminaries}\label{sec-prelim}

\subsection{The contact Euclidean space as the Heisenberg group}

We consider the contact Euclidean space $(\mathbb{R}^{2n+1},\eta,\Phi,\xi,\overline g)$ with
coordinates $(x_{1},\dots,x_{n},y_{1},\dots,y_{n},z)$, contact $1$-form
\begin{equation}\label{eqn-LeM1}
\eta =\frac{1}{2}dz-\frac{1}{4}\sum_{i=1}^{n}(y_{i}dx_{i}-x_{i}dy_{i}),
\end{equation}
Reeb vector field $\xi =2\frac{\partial }{\partial z}$, associated metric
\begin{equation*}
\overline g=\frac{1}{4}\sum_{i=1}^{n}(dx_{i}^{2}+dy_{i}^{2})+\eta \otimes \eta
\end{equation*}
and tensor
\begin{equation*}
\Phi =\sum_{i=1}^{n}\Bigl(-dx_{i}\otimes \frac{\partial }{\partial y_{i}}
+dy_{i}\otimes \frac{\partial }{\partial x_{i}}+\frac{x_{i}}{2}dx_{i}\otimes
\frac{\partial }{\partial z}+\frac{y_{i}}{2}dy_{i}\otimes \frac{\partial }{\partial z}\Bigr).
\end{equation*}
We write $\la\cdot,\cdot\ra$ for $\overline g$ and $\on$ for its Levi-Civita connection, and
we use the global frame $\{E_{A}\}_{A=1}^{2n+1}$ given by
\begin{equation}\label{eqn-frame}
E_{i}=2\frac{\partial }{\partial x_{i}}+y_{i}\frac{\partial }{\partial z},\qquad
E_{n+i}=-2\frac{\partial }{\partial y_{i}}+x_{i}\frac{\partial }{\partial z},\qquad
E_{2n+1}=\xi=2\frac{\partial }{\partial z}
\end{equation}
for $i=1,\dots,n$; equivalently,
\begin{equation}\label{eqn-LeM2}
\frac{\p}{\p x_i} =\frac{1}{2}E_i - \frac{1}{4}y_i\, \xi, \qquad
\frac{\p}{\p y_i} = -\frac{1}{2}E_{n+i} + \frac{1}{4}x_i\,\xi, \qquad
\frac{\p}{\p z} =\frac{1}{2}\xi .
\end{equation}
From these formulas one reads off that the frame is orthonormal, that
$\eta(E_a)=0$ for $a\leq 2n$, and that
\begin{equation}\label{eqn-Phi}
\Phi E_i=E_{n+i},\qquad \Phi E_{n+i}=-E_i,\qquad \Phi\xi=0 ;
\end{equation}
hence $\Phi^{2}=-I+\eta\otimes\xi$ and
$\la\Phi U,\Phi W\ra=\la U,W\ra-\eta(U)\eta(W)$, so that $(\Phi,\xi,\eta,\overline g)$ is an
almost contact metric structure. Moreover $d\eta=\frac12\sum_idx_i\wedge dy_i$, which in
the frame reads
\begin{equation}\label{eqn-deta}
d\eta(U,W)=2\la U,\Phi W\ra \qquad\text{for all vector fields } U,W .
\end{equation}

The structure is best understood through the group law
\begin{equation}\label{eqn-grouplaw}
(x,y,z)\cdot(x',y',z')=\Bigl(x+x',\ y+y',\ z+z'+\frac12(y\cdot x'-x\cdot y')\Bigr),
\end{equation}
which makes $\mathbb R^{2n+1}$ into the Heisenberg group $\mathbb H^n$. We write
$\mathcal D:=\ker\eta$ for the contact distribution, which is spanned by
$E_1,\dots,E_{2n}$, and $\Pi:\mathbb R^{2n+1}\to\mathbb C^n$, $\Pi(x,y,z)=x+\sqrt{-1}\,y$,
for the projection along the $z$-axis. The facts about the ambient space which we shall use are collected
in the following proposition. An equivalent statement for $n=2$, proved by verifying the
defining identities of a Sasakian structure and computing the curvature in coordinates,
is contained in \cite{CWZ24}; the proof below uses instead the left-invariance of the
frame.

\begin{pro}\label{pro-heis}
\begin{itemize}
\item[(i)] The $1$-form $\eta$, the tensor $\Phi$, the metric $\overline g$ and the frame
\eqref{eqn-frame} are invariant under the left translations of $\mathbb H^n$, and the only
non-vanishing brackets of the frame are
\begin{equation}\label{eqn-brackets}
[E_i,E_{n+i}]=2\xi=-[E_{n+i},E_i],\qquad i=1,\dots,n .
\end{equation}
\item[(ii)] The Levi-Civita connection of\/ $\overline g$ is given by
\begin{equation}\label{eqn-covariant}
\on_{E_i}E_{n+i}=\xi=-\on_{E_{n+i}}E_i,\qquad
\on_{E_a}\xi=\on_{\xi}E_a=-\Phi E_a \quad(a\leq 2n),
\end{equation}
all other $\on_{E_A}E_B$ being zero. In particular
\begin{equation}\label{eqn-nablaxi}
\on_U\xi=-\Phi U \qquad\text{for every vector field } U .
\end{equation}
\item[(iii)] $(\mathbb R^{2n+1},\eta,\Phi,\xi,\overline g)$ is a Sasakian manifold:
\begin{equation}\label{eqn-sasaki}
(\on_U\Phi)W=\la U,W\ra\,\xi-\eta(W)\,U .
\end{equation}
\item[(iv)] $\Pi$ is a Riemannian submersion onto $\mathbb C^n$ with the flat metric
$\frac14\sum_i(dx_i^2+dy_i^2)$, with horizontal distribution $\mathcal D$, and
\begin{equation}\label{eqn-curv}
\overline K(U,W)=-3\la\Phi U,W\ra^{2}
\end{equation}
for orthonormal $U,W\in\mathcal D$. In particular the $\Phi$-sectional curvature
$\overline K(U,\Phi U)$ is identically $-3$, while $\overline K(U,W)=0$ whenever
$\la\Phi U,W\ra=0$. The map 
\begin{equation*}
(x,y,z)\mapsto(x,y,z+\frac12x\cdot y)
\end{equation*}
 is an isomorphism of
Sasakian structures onto the standard model $\mathbb R^{2n+1}(-3)$ of\/ \cite{Bla02}, whose
contact form is $\frac12(dz-\sum_iy_idx_i)$.
\end{itemize}
\end{pro}

\begin{proof}
(i) Left translation by $(x^0,y^0,z^0)$ is

\begin{equation*}
\ell(x,y,z)=\bigl(x+x^0,\,y+y^0,\,z+z^0+\frac12(y^0\cdot x-x^0\cdot y)\bigr) ,
\end{equation*}
 so that
\begin{equation*}
\ell_*\frac{\p}{\p x_i}=\frac{\p}{\p x_i}+\frac{y^0_i}2\frac{\p}{\p z},\qquad
\ell_*\frac{\p}{\p y_i}=\frac{\p}{\p y_i}-\frac{x^0_i}2\frac{\p}{\p z},\qquad
\ell_*\frac{\p}{\p z}=\frac{\p}{\p z} .
\end{equation*}
Hence $\ell_*(E_i|_p)=2\frac{\p}{\p x_i}+(y_i+y^0_i)\frac{\p}{\p z}=E_i|_{\ell(p)}$, and in the
same way $\ell_*(E_{n+i}|_p)=E_{n+i}|_{\ell(p)}$ and $\ell_*\xi=\xi$. Since $\eta$, $\Phi$ and
$\overline g$ have constant coefficients with respect to the frame and its dual coframe,
they are left-invariant as well. The brackets \eqref{eqn-brackets} follow at once from
\eqref{eqn-frame}: the only non-constant coefficients are $y_i$ in $E_i$ and $x_i$ in
$E_{n+i}$, and

\begin{equation*}
[E_i,E_{n+i}]=2\frac{\p x_i}{\p x_i}\frac{\p}{\p z}+2\frac{\p y_i}{\p y_i}\frac{\p}{\p z}
=4\frac{\p}{\p z}=2\xi .
\end{equation*}

(ii) For an orthonormal frame the Koszul formula reduces to
\begin{equation*}
2\la\on_{E_A}E_B,E_C\ra=\la[E_A,E_B],E_C\ra-\la[E_B,E_C],E_A\ra+\la[E_C,E_A],E_B\ra .
\end{equation*}
By \eqref{eqn-brackets} the right-hand side vanishes unless $\{A,B,C\}=\{i,n+i,2n+1\}$
for some $i$, and in that case it gives
$2\la\on_{E_i}E_{n+i},\xi\ra=2$, $2\la\on_{E_i}\xi,E_{n+i}\ra=-2$ and
$2\la\on_{E_{n+i}}\xi,E_{i}\ra=2$. Since $[E_a,\xi]=0$ we have $\on_\xi E_a=\on_{E_a}\xi$.
This is \eqref{eqn-covariant}, and \eqref{eqn-nablaxi} follows because both sides are
tensorial in $U$ and $\on_\xi\xi=0=\Phi\xi$.

(iii) By \cite[Theorem 6.3]{Bla02} an almost contact metric structure is Sasakian if and
only if \eqref{eqn-sasaki} holds. Both sides of \eqref{eqn-sasaki} are tensorial, so it
suffices to take $U,W$ in the frame, and since $\Phi$ has constant coefficients there,
$(\on_U\Phi)W=\on_U(\Phi W)-\Phi(\on_UW)$ is computed directly from \eqref{eqn-Phi} and
\eqref{eqn-covariant}: for $i,j\leq n$,
\begin{equation*}\aligned
&(\on_{E_i}\Phi)E_j=\on_{E_i}E_{n+j}=\d_{ij}\xi,\qquad
(\on_{E_{n+i}}\Phi)E_{n+j}=-\on_{E_{n+i}}E_{j}=\d_{ij}\xi,\\
&(\on_{E_a}\Phi)\xi=-\Phi(-\Phi E_a)=-E_a ,
\endaligned\end{equation*}
while $(\on_{E_i}\Phi)E_{n+j}=-\on_{E_i}E_j-\d_{ij}\Phi\xi=0$,
$(\on_{E_{n+i}}\Phi)E_{j}=\on_{E_{n+i}}E_{n+j}+\d_{ij}\Phi\xi=0$ and
$(\on_\xi\Phi)E_a=-\Phi^2E_a+\Phi^2E_a=0=(\on_\xi\Phi)\xi$, in agreement with the right-hand side of
\eqref{eqn-sasaki} in each case.

(iv) By \eqref{eqn-frame}, $\Pi_*E_i=2\frac{\p}{\p x_i}$, $\Pi_*E_{n+i}=-2\frac{\p}{\p y_i}$
and $\Pi_*\xi=0$; thus $\Pi_*$ maps $\mathcal D$ isometrically onto
$(T\mathbb C^n,\frac14\sum_i(dx_i^2+dy_i^2))$ and $\ker\Pi_*=\mathbb R\xi$. O'Neill's
formula \cite{ONe66} for orthonormal horizontal vectors gives
$\overline K(U,W)=K^{\mathbb C^n}(\Pi_*U,\Pi_*W)-\frac34\bigl|[U,W]^{\mathrm{ver}}\bigr|^{2}$, where
$U,W$ are extended to basic vector fields and $(\cdot)^{\mathrm{ver}}$ is the vertical component. The base is flat,
and for $U,W$ in the span of $E_1,\dots,E_{2n}$ with constant coefficients,
which are basic, \eqref{eqn-brackets} and \eqref{eqn-Phi} give
$[U,W]=2\la\Phi U,W\ra\xi$. This proves \eqref{eqn-curv}. Finally, the map $(x,y,z)\mapsto(x,y,z+\frac12x\cdot y)$ pulls
$\frac12(dz-\sum_iy_idx_i)$ back to $\eta$, and fixes $\xi$ and the coordinates $x,y$; hence
it pulls the metric of the standard model back to $\overline g$, and since $\Phi$ is
determined by $\overline g$ and $d\eta$ through \eqref{eqn-deta}, it intertwines the two
structures.
\end{proof}

\subsection{Legendrian submanifolds and the flow}

Let $F:\Sigma^n\to\mathbb R^{2n+1}$ be a Legendrian immersion, that is, $F^*\eta=0$. We
denote by $\n$ the Levi-Civita connection of the induced metric, by $(\cdot)^T$ and
$(\cdot)^\perp$ the orthogonal projections onto the tangent bundle $T\Sigma^n$ and the
normal bundle $N\Sigma^n$, by $B(U,W)=(\on_UW)^\perp$ the second fundamental form and by
$H=\mathrm{trace}\,B$ the mean curvature vector. Since $F^*d\eta=0$, \eqref{eqn-deta}
gives $\la U,\Phi W\ra=0$ for $U,W\in T\Sigma^n$, so that
\begin{equation}\label{eqn-normal}
N\Sigma^n=\Phi(T\Sigma^n)\oplus\mathbb R\,\xi ,
\end{equation}
and if $\{e_i\}$ is a local orthonormal frame of $\Sigma^n$ then
$\{e_i,\Phi e_i,\xi\}$ is an orthonormal frame of\/ $\mathbb R^{2n+1}$ along $\Sigma^n$
adapted to it. By \eqref{eqn-nablaxi},
$\la B(U,W),\xi\ra=-\la W,\on_U\xi\ra=\la W,\Phi U\ra=0$, that is,
\begin{equation}\label{eqn-Bxi}
B \text{ takes its values in } \Phi(T\Sigma^n),\qquad\text{in particular } \la H,\xi\ra=0 .
\end{equation}
Since $\overline g$ is a left-invariant metric on the Lie group $\mathbb H^n$, it is
complete, and a properly immersed submanifold is therefore complete in its induced metric;
we shall nevertheless say ``complete properly immersed'' when both properties are used.

The \emph{Legendrian angle} $\theta$ of $F$ is defined as follows. Identify $E_j$ and
$E_{n+j}=\Phi E_j$ with the $j$-th real and imaginary standard basis vectors of
$\mathbb C^n$; if $v_1,\dots,v_n$ is an oriented, not necessarily orthonormal, tangent frame and
${\det}_{\mathbb C}$ is the determinant of the resulting complex $n\times n$ matrix, then
\begin{equation}\label{eqn-angledet}
{\det}_{\mathbb C}(v_1,\dots,v_n)=e^{\sqrt{-1}\,\theta}\sqrt{\det\bigl(\la v_i,v_j\ra\bigr)} .
\end{equation}
This determines $\theta$ modulo $2\pi$, and reversing the orientation of\/ $\Sigma^n$ changes
it by $\pi$. All submanifolds are assumed to be connected, oriented and without boundary, and whenever $\theta$
itself, and not only $\n\theta$, enters an equation, it is understood that $\theta$ is a
single-valued real function on $\Sigma^n$; this is automatic when $\Sigma^n$ is simply
connected. Thus the branch and the orientation are part of the data of such an equation,
although the mean curvature vector \eqref{eqn-LeSh3} does not depend on either. It is the Legendre angle of Smoczyk \cite{S03}, denoted by $\alpha$ in
\cite{CWZ24}, and it satisfies
\begin{equation}\label{eqn-LeSh3}
H=\Phi\n\theta .
\end{equation}

The following elementary lemma describes the normal variations which are tangent to the
space of Legendrian immersions; cf.\ \cite[Lemma 3.1]{S03}.

\begin{lem}\label{lem-ham}
Let $F:\Sigma^n\to\mathbb R^{2n+1}$ be a Legendrian immersion, and let $F_t$ be a
variation of $F=F_0$ with normal variation field
$\frac{\p F_t}{\p t}\big|_{t=0}=\Phi W+h\,\xi$, where $W$ is a tangent vector field
of\/ $\Sigma^n$ and $h\in C^\infty(\Sigma^n)$. Then
\begin{equation*}
\frac{\p}{\p t}\Big|_{t=0}F_t^{*}\eta=dh+2\,W^{\flat} .
\end{equation*}
Hence the variation preserves the Legendrian condition to first order if and only if its
variation field is
\begin{equation}\label{eqn-Vh}
\mathcal V_h:=-\frac12\,\Phi\n h+h\,\xi
\end{equation}
for some $h\in C^\infty(\Sigma^n)$, and $h=\eta(\mathcal V_h)$ determines $\mathcal V_h$.
\end{lem}

\begin{proof}
Put $V=\Phi W+h\xi$. By Cartan's formula and \eqref{eqn-deta}, for $Y\in T\Sigma^n$,
\begin{equation*}
\frac{\p}{\p t}\Big|_{t=0}\bigl(F_t^{*}\eta\bigr)(Y)=d\bigl(\eta(V)\bigr)(Y)+d\eta(V,Y)
=Y(h)+2\la\Phi W,\Phi Y\ra=Y(h)+2\la W,Y\ra ,
\end{equation*}
because $\la\xi,\Phi Y\ra=0$ and $\Phi$ is an isometry of\/ $\mathcal D$.
\end{proof}

By \eqref{eqn-LeSh3} and \eqref{eqn-Bxi}, the mean curvature vector $H=\Phi\n\theta$ is not
of the form \eqref{eqn-Vh} unless $\theta$ is constant, whereas
\begin{equation}\label{eqn-Hvar}
H-2\theta\,\xi=\mathcal V_{-2\theta} .
\end{equation}
This is the velocity of the Legendrian mean curvature flow \eqref{eqn-int-lmcf}
of Smoczyk \cite{S03}. By Lemma \ref{lem-ham} a vector field of the form $H+h\,\xi$ is
tangent to the space of Legendrian immersions if and only if $h=-2\theta$ up to an additive
constant, and a constant corresponds to a translation along $\xi$.

\subsection{Entire Legendrian graphs}

\begin{lem}\label{lem-graph}
Let $\Sigma^n=\{(x,y(x),z(x)):x\in\mathbb R^n\}$ be the graph of a smooth map
$(y,z):\mathbb R^n\to\mathbb R^n\times\mathbb R$. Then $\Sigma^n$ is Legendrian if and only
if there is a smooth function $u:\mathbb R^n\to\mathbb R$ such that
\begin{equation}\label{eqn-graph}
y=Du,\qquad z=u-\frac12\,x\cdot Du ,
\end{equation}
and then $u=z+\frac12x\cdot y$ is unique.
\end{lem}

\begin{proof}
By \eqref{eqn-LeM1}, $\eta|_{\Sigma^n}=0$ means 
\begin{equation*}
dz=\frac12\sum_i(y_idx_i-x_idy_i) ,
\end{equation*}
 that is,
$d\bigl(z+\frac12x\cdot y\bigr)=\sum_iy_idx_i$. Thus $u:=z+\frac12x\cdot y$ satisfies
$Du=y$, and \eqref{eqn-graph} follows; the converse is the same computation read
backwards.
\end{proof}

Let $F=F_u$ be the immersion \eqref{eqn-int-graph}. We write $u_i,u_{ij},\dots$
for the partial derivatives of $u$ and sum over repeated indices. By \eqref{eqn-LeM2},
\begin{equation*}\aligned
v_i:=F_*\frac{\p}{\p x_i}
=&\ \frac{\p}{\p x_i}+u_{ki}\frac{\p}{\p y_k}+\frac12\bigl(u_i-x_ku_{ki}\bigr)\frac{\p}{\p z}
=\frac12E_i-\frac12u_{ki}E_{n+k},\\
\Phi v_i=&\ \frac12E_{n+i}+\frac12u_{ki}E_k ,
\endaligned\end{equation*}
the $\xi$-components cancelling because $F$ is Legendrian; the vectors $v_i$, $\Phi v_i$
and $\xi$ form the (non-orthonormal) adapted frame of the graph. The induced metric and
volume element are
\begin{equation}\label{eqn-metric}
g_{ij}=\la v_i,v_j\ra=\frac14\bigl(\d_{ij}+u_{ki}u_{kj}\bigr),\qquad
d\mu=2^{-n}\prod_{i=1}^{n}\bigl(1+\lambda_i^{2}\bigr)^{1/2}dx ,
\end{equation}
where $\lambda_1,\dots,\lambda_n$ are the eigenvalues of $D^2u$. In particular
$g_{ij}\geq\frac14\d_{ij}$, so that an entire Legendrian graph is complete. Under the
identification preceding \eqref{eqn-angledet}, $v_j$ corresponds to the $j$-th column of
the matrix $\frac12(I-\sqrt{-1}D^2u)$, whose determinant is
$2^{-n}\prod_k(1-\sqrt{-1}\lambda_k)$. Comparing with \eqref{eqn-angledet} and
\eqref{eqn-metric} we obtain
\begin{equation}\label{eqn-thetagraph}
\theta=-\sum_{i=1}^{n}\arctan\lambda_i ,\qquad |\theta|<\frac{n\pi}2 .
\end{equation}
The minus sign comes from the minus sign in $E_{n+i}=-2\frac{\p}{\p y_i}+x_i\frac{\p}{\p z}$.

\begin{rem}
For a family of graphs $F_{u(\cdot,t)}$ one has
$\eta\bigl(\frac{\p F}{\p t}\bigr)=\frac12\frac{\p u}{\p t}$, so that by Lemma \ref{lem-ham}
and \eqref{eqn-Hvar} the Legendrian mean curvature flow \eqref{eqn-int-lmcf} of entire
graphs is, up to tangential diffeomorphisms, the scalar equation
\begin{equation*}
\frac{\p u}{\p t}=-4\theta=4\sum_{i=1}^{n}\arctan\lambda_i .
\end{equation*}
\end{rem}

\vskip24pt

\section{Contact dilations and the self-shrinker equation}\label{sec-motivation}

Two features of the self-shrinker equation \eqref{eqn-int-shrinker} may look ad hoc at
first sight: the velocity on the left-hand side is $H-2\theta\xi$ rather than $H$, and the
vector field $X$ on the right-hand side is not the position vector, its
$\frac{\p}{\p z}$-component being $2z$ rather than $z$. Neither is a matter of choice. The
first is dictated by the flow, as we saw in \eqref{eqn-Hvar}; the second is dictated by
the symmetry group of the flow, as we now explain. The discussion is parallel to that of
the translating field in \cite{CQZ26}.

\subsection{The Euclidean model: the position vector as the generator of the dilations}

Suppose that a geometric flow of submanifolds, $(\frac{\p F}{\p t})^{\perp}=\mathcal W(F)$,
is invariant under a one-parameter group $\{\tau_\sigma\}$ of transformations of the
ambient space, with infinitesimal generator $\mathcal G$. A soliton is a solution which moves by
that group, $F(\cdot,t)=\tau_{\sigma(t)}\circ F(\cdot,0)$ up to tangential
diffeomorphisms, and differentiating in $t$ one finds that the soliton equation is
\begin{equation*}
\mathcal W(F)=\sigma'(t)\,\mathcal G^{\perp} .
\end{equation*}
Thus the datum which produces a soliton equation is the generator of a group of symmetries
of the flow, and nothing else. For the mean curvature flow in $\mathbb R^m$ the
translations give the translating solitons $H=\mathcal G^{\perp}$, with $\mathcal G$ a constant vector, and
the dilations $p\mapsto\lambda p$ give the self-shrinkers: their generator is the Euler
field $P$, and $F(\cdot,t)=\sqrt{-t}\,F(\cdot,-1)$ leads to $H=\frac1{2t}P^{\perp}$.

That the Euler field coincides with the position vector uses the affine and metric
structures of Euclidean space. The affine structure identifies
$T_p\mathbb R^m$ with $\mathbb R^m$, so that a point and the vector pointing at it are the
same object; and, more substantially, the dilations are symmetries of the mean curvature
flow because they are homotheties of the Euclidean metric, which is the content of the
identity $\on P=\mathrm{Id}$. Thus the position vector in the self-shrinker equation can
also be described as the generator of the dilation subgroup of the symmetry group of the
flow. This description extends to the contact setting.

\subsection{The Euclidean dilations are not symmetries of the Legendrian mean curvature flow}

The flow \eqref{eqn-int-lmcf} is not merely a mean curvature flow: it is a flow in the
class of Legendrian submanifolds, that is, of the submanifolds constrained by $\eta$. A
symmetry of this flow must therefore preserve the contact distribution
$\mathcal D=\ker\eta$. The Euclidean dilation
$\d^{\mathrm{euc}}_\lambda(x,y,z)=(\lambda x,\lambda y,\lambda z)$ does not: it multiplies the term $dz$ of $\eta$ by $\lambda$ and the terms
$y_idx_i-x_idy_i$ by $\lambda^{2}$. Infinitesimally, for the Euclidean position vector field
$P=\sum_i(x_i\frac{\p}{\p x_i}+y_i\frac{\p}{\p y_i})+z\frac{\p}{\p z}$ of\/
$\mathbb R^{2n+1}$ one computes from \eqref{eqn-LeM1}
\begin{equation*}
L_P\eta=\frac12dz-\frac12\sum_{i=1}^n(y_idx_i-x_idy_i),
\end{equation*}
which is not a multiple of $\eta$. Hence $P$ is not a contact vector field and its flow
does not map the class of Legendrian submanifolds to itself, although it does preserve some
individual Legendrian submanifolds, for instance the horizontal $n$-plane
$\{(x,0,0)\}$. Consequently $\{\d^{\mathrm{euc}}_\lambda\}$ is not a group of symmetries of
\eqref{eqn-int-lmcf}, and an equation with $P^{\perp}$ on the right-hand side, whatever the
constant in front of it, does not express self-similarity with respect to a symmetry of the
flow.

\subsection{The contact dilations}

By Proposition \ref{pro-heis} (i) the Lie algebra of left-invariant vector fields of\/
$\mathbb H^n$ is graded,
\begin{equation*}
\mathfrak h^n=\mathfrak v_1\oplus\mathfrak v_2,\qquad
\mathfrak v_1=\mathrm{span}\{E_1,\dots,E_{2n}\},\qquad
\mathfrak v_2=\mathbb R\,\xi,\qquad [\mathfrak v_1,\mathfrak v_1]=\mathfrak v_2 ,
\end{equation*}
and an automorphism which acts as $\lambda$ on $\mathfrak v_1$ must act as $\lambda^{2}$ on
$\mathfrak v_2$. Hence, among the automorphisms which act on the whole horizontal layer by
one and the same scalar $\lambda$, the only ones are the contact dilations $\d_\lambda$ of
\eqref{eqn-int-dil}, which satisfy
\begin{equation}\label{eqn-dil}
\d_\lambda^{*}\eta=\lambda^{2}\eta,\qquad
(\d_\lambda)_*E_a=\lambda\,E_a\circ\d_\lambda\quad(a\leq2n),\qquad
(\d_\lambda)_*\xi=\lambda^{2}\,\xi\circ\d_\lambda ,
\end{equation}
as one checks from \eqref{eqn-LeM1} and \eqref{eqn-frame}. There is an equivalent and
perhaps more transparent reason for the weight $2$: by Proposition \ref{pro-heis} (iv),
$\eta$ is a connection form for the fibration $\Pi:\mathbb R^{2n+1}\to\mathbb C^n$, and its
curvature $d\eta=\frac12\sum_idx_i\wedge dy_i$ is a $2$-form on the base; scaling the base
by $\lambda$ therefore scales the curvature by $\lambda^{2}$, and the fibre coordinate $z$
must scale by $\lambda^{2}$ as well. By \eqref{eqn-LeM2} the generator
\eqref{eqn-int-X} of\/ $\{\d_\lambda\}$ is
\begin{equation}\label{eqn-X}
X=\frac12\sum_{i=1}^n\bigl(x_iE_i-y_iE_{n+i}\bigr)+z\,\xi,\qquad
\eta(X)=z,\qquad L_X\eta=2\eta ,
\end{equation}
so that $X=P+\frac z2\,\xi$ is a contact vector field.

The dilation $\d_\lambda$ is an automorphism of the group \eqref{eqn-grouplaw} and of the
contact structure, but it is not a homothety of\/ $\overline g$:
\begin{equation*}
\d_\lambda^{*}\overline g=\lambda^{2}\overline g+(\lambda^{4}-\lambda^{2})\,\eta\otimes\eta
\end{equation*}
is the $\mathcal D$-homothetic deformation of Tanno \cite{Tan68} with parameter
$\lambda^{2}$. This is sufficient for our purposes, because the flow only sees the geometry
of\/ $\mathcal D$: the second fundamental form of a Legendrian submanifold takes its values
in $\mathcal D$ by \eqref{eqn-Bxi}, on $\mathcal D$ one has
$\d_\lambda^{*}\overline g=\lambda^{2}\overline g$, and by \eqref{eqn-dil} the Legendrian
angle is invariant under $\d_\lambda$. The price to pay is that $X$ is not a homothetic
vector field, $\on X\neq\mathrm{Id}$; this is the origin of the additional terms in Lemma
\ref{lem-DX} and of the function $f$ in Lemma \ref{lem1}.

\subsection{Legendrian self-shrinkers}

\begin{lem}\label{lem-Xperp}
Let $F:\Sigma^n\to\mathbb R^{2n+1}$ be a Legendrian immersion and denote by $z$ also the
restriction $z\circ F$ of the last coordinate to $\Sigma^n$. Then
\begin{equation}\label{eqn-Xperp}
\la X,\Phi Y\ra=-\frac12\,Y(z)\quad\text{for } Y\in T\Sigma^n,\qquad\text{that is,}\qquad
X^{\perp}=-\frac12\,\Phi\n z+z\,\xi=\mathcal V_z .
\end{equation}
\end{lem}

\begin{proof}
From 
\[
L_X\eta=2\eta \quad \text{and}  \quad \eta(X)=z
\]
 we get 
 \[
 \iota_Xd\eta=2\eta-dz.
 \]
  For $Y\in T\Sigma^n$ we
have $\eta(Y)=0$, so \eqref{eqn-deta} gives 
\[
2\la X,\Phi Y\ra=d\eta(X,Y)=-Y(z).
\]
 Since
$\la X,\xi\ra=z$, the second formula follows from \eqref{eqn-normal}.
\end{proof}

Thus $X^{\perp}$ is a Legendrian variation in the sense of Lemma \ref{lem-ham}, as the
normal part of the generator of a group of contact transformations must be. For the
Euclidean position vector field $P=X-\frac z2\xi$ one obtains instead
$P^{\perp}=-\frac12\Phi\n z+\frac z2\xi$, which is not of the form \eqref{eqn-Vh} unless $z$
is constant: its vertical part is half of what its horizontal part requires. This displays
the difference between the two vector fields directly, $\eta(X)=z=2\,\eta(P)$.

\begin{defi}\label{def-shrinker}
A Legendrian immersion $F:\Sigma^n\to\mathbb R^{2n+1}$ is a \emph{Legendrian self-shrinker}
if
\begin{equation}\label{eqn-LeSh}
H-2\theta\,\xi=-\frac12\,X^{\perp},
\end{equation}
where $X$ is the generator \eqref{eqn-X} of the contact dilations.
\end{defi}

\begin{pro}\label{pro-scalar}
Let $F:\Sigma^n\to\mathbb R^{2n+1}$ be a Legendrian immersion.
\begin{itemize}
\item[(i)] $F$ satisfies \eqref{eqn-LeSh} if and only if
\begin{equation}\label{eqn-scalar}
4\theta=z \qquad\text{on } \Sigma^n .
\end{equation}
\item[(ii)] If\/ $F$ satisfies \eqref{eqn-LeSh}, then $F_t:=\d_{\sqrt{-t}}\circ F$,
$t\in(-\infty,0)$, is a family of Legendrian immersions with
\begin{equation*}
\Bigl(\frac{\p F_t}{\p t}\Bigr)^{\perp}=H_t-2\theta_t\,\xi=\frac1{2t}\,X^{\perp} ,
\end{equation*}
that is, a solution of the Legendrian mean curvature flow \eqref{eqn-int-lmcf} up to
tangential diffeomorphisms, with $F_t(p)\to0$ as $t\nearrow0$ for every $p\in\Sigma^n$.
\end{itemize}
\end{pro}

\begin{proof}
(i) By \eqref{eqn-Hvar} and \eqref{eqn-Xperp}, equation \eqref{eqn-LeSh} reads
$\mathcal V_{-2\theta}=-\frac12\mathcal V_{z}=\mathcal V_{-z/2}$, and $h\mapsto\mathcal V_h$
is linear and injective by Lemma \ref{lem-ham}.

(ii) Put $\lambda=\sqrt{-t}$. Since $\d_\lambda^*\eta=\lambda^2\eta$, $F_t$ is Legendrian.
By \eqref{eqn-dil}, $(\d_\lambda)_*$ multiplies the complex determinant
\eqref{eqn-angledet} by $\lambda^{n}>0$, so that $\theta_t=\theta$, while
$z\circ F_t=\lambda^{2}z$. Hence \eqref{eqn-scalar} gives
\[
2\theta_t-\frac1{2\lambda^{2}}\,z\circ F_t=0,
\]
 and the argument of (i) shows that
\[
H_t-2\theta_t\xi=-\frac1{2\lambda^{2}}X^{\perp}=\frac1{2t}X^{\perp} \quad {\text{along}}  \quad F_t.
\]
 On the
other hand $\frac{d}{d\lambda}\d_\lambda=\lambda^{-1}X\circ\d_\lambda$, whence
$\frac{\p F_t}{\p t}=\frac{1}{\lambda}\frac{d\lambda}{dt}\,X\circ F_t=\frac1{2t}\,X\circ F_t$.
\end{proof}

Thus \eqref{eqn-LeSh} is the counterpart of the Euclidean self-shrinker equation
$H=-\frac12P^{\perp}$, with $H$ replaced by the velocity of the flow which preserves the
Legendrian class, and the position vector replaced by the generator of the dilations which
preserve the contact structure. In particular the components of \eqref{eqn-LeSh} along
$\Phi(T\Sigma^n)$ and along $\xi$ are not two conditions but one: the former is the
differential of the latter. For entire graphs, combining \eqref{eqn-graph},
\eqref{eqn-thetagraph} and Proposition \ref{pro-scalar} we arrive at the following
statement.

\begin{pro}\label{pro-pde}
The entire Legendrian graph $F_u$ is a Legendrian self-shrinker if and only if
\begin{equation}\label{eqn-LeSh2}
\sum_{i=1}^{n}\arctan \lambda _{i}(x)= \frac{1}{4} \Bigl(\frac{1}{2}\,x \cdot Du(x)-u(x)\Bigr),
\qquad x\in\mathbb R^n .
\end{equation}
In this case the solution of Proposition \ref{pro-scalar} (ii) is the graph of
$u(x,t)=-t\,u\bigl(x/\sqrt{-t}\bigr)$.
\end{pro}

Indeed $\d_\lambda(x,Du,u-\frac12x\cdot Du)=(\lambda x,\lambda Du,\lambda^{2}(u-\frac12x\cdot Du))$
is the point of the graph of $u_\lambda:=\lambda^{2}u(\cdot/\lambda)$ over $\lambda x$.

\begin{rem}
The sign of the right-hand side of \eqref{eqn-LeSh2} is essential. With the opposite sign
one obtains the equation of the self-expanders $u(x,t)=t\,u(x/\sqrt t)$, $t>0$, for which
no rigidity holds. Already for $n=1$, if $\arctan u''=\frac14(u-\frac12xu')$ then
$z=u-\frac12xu'$ satisfies $z''=-\frac x8(1+u''^{2})z'$, so that
$|z'(x)|\leq\frac12|u'(0)|e^{-x^{2}/16}$; hence every solution with
$|u(0)|+\sqrt\pi\,|u'(0)|<2\pi$ is entire, and it is not a quadratic polynomial when
$u'(0)\neq0$. For \eqref{eqn-LeSh2} itself the sign of the right-hand side of the equation
for $z''$ is reversed, and $z'$ grows instead of decaying.
\end{rem}

\vskip24pt

\section{Basic identities and weighted volume estimates}\label{sec-vol}

The Euclidean analysis of self-shrinkers rests on the fact that the position vector field
$P$ of\/ $\mathbb R^N$ satisfies $\on P=\mathrm{Id}$. The generator $X$ of the contact
dilations is not a homothetic vector field of\/ $\overline g$, and the first step is to
find what replaces this identity along a Legendrian submanifold. Throughout the remaining
sections, $z$ denotes the restriction to $\Sigma^n$ of the last coordinate, as in Lemma
\ref{lem-Xperp}, and $X^{\mathcal H}=X-z\,\xi=\frac12\sum_i(x_iE_i-y_iE_{n+i})$ the
horizontal part of $X$.

\begin{lem}\label{lem-DX}
For every $U\in\mathcal D$,
\begin{equation}\label{eqn-DXU}
\on_UX=U-z\,\Phi U+\la U,\Phi X\ra\,\xi .
\end{equation}
Consequently, if\/ $F:\Sigma^n\to\mathbb R^{2n+1}$ is a Legendrian immersion and
$Y\in T\Sigma^n$, then
\begin{equation}\label{eqn-DX}
\on_YX=Y-z\,\Phi Y+\frac12\,Y(z)\,\xi .
\end{equation}
\end{lem}

\begin{proof}
Both sides of \eqref{eqn-DXU} are tensorial in $U$, so it suffices to take $U=E_j$ and
$U=E_{n+j}$. By \eqref{eqn-frame}, $E_j(x_i)=2\d_{ij}$, $E_j(y_i)=0$ and $E_j(z)=y_j$, so
that \eqref{eqn-X} and \eqref{eqn-covariant} give
\begin{equation*}\aligned
\on_{E_j}X=&\ \frac12E_j(x_i)E_i+E_j(z)\,\xi-\frac12y_i\,\on_{E_j}E_{n+i}+z\,\on_{E_j}\xi\\
=&\ E_j+y_j\,\xi-\frac12y_j\,\xi-z\,E_{n+j}
=E_j-z\,\Phi E_j+\frac12y_j\,\xi ,
\endaligned\end{equation*}
and $\la E_j,\Phi X\ra=\frac12y_j$ because
$\Phi X=\frac12\sum_i(x_iE_{n+i}+y_iE_i)$. The computation for $E_{n+j}$ is the same,
with $E_{n+j}(y_i)=-2\d_{ij}$, $E_{n+j}(z)=x_j$, $\on_{E_{n+j}}E_i=-\d_{ij}\xi$,
$\on_{E_{n+j}}\xi=E_j=-\Phi E_{n+j}$ and $\la E_{n+j},\Phi X\ra=\frac12x_j$. Finally, for
$Y\in T\Sigma^n$ we have $\la Y,\Phi X\ra=-\la\Phi Y,X\ra=\frac12Y(z)$ by
\eqref{eqn-Xperp}, which turns \eqref{eqn-DXU} into \eqref{eqn-DX}.
\end{proof}

\begin{lem}\label{lem1}
Let $F:\Sigma^n\to\mathbb R^{2n+1}$ be a Legendrian immersion and let
\begin{equation}\label{eqn-f}
f:=\frac14\bigl|X^{\mathcal H}\bigr|^{2}=\frac14\bigl(|X|^{2}-z^{2}\bigr)
=\frac1{16}\sum_{i=1}^n\bigl(x_i^{2}+y_i^{2}\bigr)\Big|_{\Sigma^n} .
\end{equation}
Then
\begin{equation}\label{eqn-Lem1}\aligned
\n f =&\ \frac{1}{2}X^T, \\
\D f =&\ \frac{n}{2} + \frac{1}{2}\la X, H\ra.
\endaligned\end{equation}
\end{lem}

\begin{proof}
Let $Y\in T\Sigma^n$. By \eqref{eqn-DX} and \eqref{eqn-Xperp},
\begin{equation*}\aligned
Y(f)=&\ \frac12\la\on_YX,X\ra-\frac12z\,Y(z)
=\frac12\Bigl(\la Y,X\ra-z\la\Phi Y,X\ra+\frac12z\,Y(z)\Bigr)-\frac12z\,Y(z)\\
=&\ \frac12\la Y,X\ra+\frac14z\,Y(z)+\frac14z\,Y(z)-\frac12z\,Y(z)=\frac12\la X,Y\ra ,
\endaligned\end{equation*}
which is the first identity. For the second, let $\{e_i\}$ be a local orthonormal frame
of\/ $\Sigma^n$. Since $\la e_i,\Phi e_i\ra=0=\la e_i,\xi\ra$, \eqref{eqn-DX} gives
$\la\on_{e_i}X,e_i\ra=1$, and therefore
\begin{equation*}\aligned
\D f=\frac12\Div\bigl(X^{T}\bigr)
=&\ \frac12\sum_{i}\Bigl(\la\on_{e_i}X,e_i\ra-\la\on_{e_i}X^{\perp},e_i\ra\Bigr)\\
=&\ \frac n2+\frac12\sum_i\la X^{\perp},\on_{e_i}e_i\ra=\frac n2+\frac12\la X,H\ra .
\endaligned\end{equation*}
\end{proof}

\begin{rem}
Lemma \ref{lem1} is the counterpart of the Euclidean identities
$\n\frac{|P|^2}4=\frac12P^{T}$ and $\D\frac{|P|^2}4=\frac n2+\frac12\la P,H\ra$, but the
function which satisfies them is $f$ and not $\frac14|X|^{2}=f+\frac14z^{2}$: by
\eqref{eqn-DX} the vertical component of $X$ contributes to $\la\on_YX,X\ra$ the
term $\frac12Y(z^{2})$ which is subtracted in \eqref{eqn-f}. For the graph
\eqref{eqn-int-graph}, $f=\frac1{16}\bigl(|x|^{2}+|Du|^{2}\bigr)$. We call $f$ the
\emph{horizontal potential} of the immersion.
\end{rem}

Following \cite{CZ13} we write $\D_f:=\D-\la\n f,\n(\cdot)\ra$; by \eqref{eqn-Lem1} this
is the drift Laplacian $\mathcal L$ of the Introduction,
\begin{equation}\label{eqn-drift}
\mathcal L=\D_f=\D-\frac12\la X,\n(\cdot)\ra=e^{f}\Div\bigl(e^{-f}\n(\cdot)\bigr) .
\end{equation}

\begin{lem}\label{lem-shrinker-identities}
On every Legendrian self-shrinker, without any completeness or properness assumption,
\begin{equation}\label{eqn-Xperp2}
X^{\perp}=-2H+z\,\xi ,
\end{equation}
and
\begin{equation}\label{eqn-fineq}
|\n f|^2=f-|H|^2\leq f,\qquad \D f=\frac n2-|H|^2,
\qquad \mathcal Lf=\frac n2-f .
\end{equation}
\end{lem}

\begin{proof}
Equation \eqref{eqn-Xperp2} follows from \eqref{eqn-LeSh} and $z=4\theta$.
Since $\la H,\xi\ra=0$ by \eqref{eqn-Bxi}, it gives
$\la X,H\ra=-2|H|^2$ and $|X^\perp|^2=4|H|^2+z^2$.
Lemma \ref{lem1} then yields
\begin{equation*}
|\n f|^2=\frac14\bigl(|X|^2-|X^\perp|^2\bigr)=f-|H|^2,
\qquad \D f=\frac n2-|H|^2 .
\end{equation*}
Subtracting $|\n f|^2$ from $\D f$ proves the last identity in \eqref{eqn-fineq}.
\end{proof}

\begin{lem}\label{lem-Dz}
Let $F:\Sigma^n\to\mathbb R^{2n+1}$ be a Legendrian immersion. Then
\begin{equation}\label{eqn-Dz}
\n z=2\,(\Phi X)^{T},\qquad \D z=2\la X,\n\theta\ra .
\end{equation}
In particular, on a Legendrian self-shrinker,
\begin{equation}\label{eqn-LeM45}
\mathcal{L}\theta = \D \theta - \frac{1}{2}\la X, \n \theta \ra = 0.
\end{equation}
Moreover,
\begin{equation}\label{eqn-height-identities}
\mathcal Lz=0,\qquad \n z=-4\Phi H,\qquad |\n z|^2=16|H|^2\leq16f .
\end{equation}
\end{lem}

\begin{proof}
The first identity is \eqref{eqn-Xperp}: $Y(z)=-2\la X,\Phi Y\ra=2\la\Phi X,Y\ra$ for
$Y\in T\Sigma^n$. Let $\{e_i\}$ be a local orthonormal frame of\/ $\Sigma^n$. By
\eqref{eqn-sasaki}, \eqref{eqn-DX} and $\Phi\xi=0$,
\begin{equation*}
\on_{e_i}(\Phi X)=(\on_{e_i}\Phi)X+\Phi\bigl(\on_{e_i}X\bigr)
=\la e_i,X\ra\,\xi-z\,e_i+\Phi e_i+z\,e_i=\la e_i,X\ra\,\xi+\Phi e_i ,
\end{equation*}
so that $\la\on_{e_i}(\Phi X),e_i\ra=0$. Therefore, as in the proof of Lemma \ref{lem1},
\begin{equation*}
\D z=2\,\Div\bigl((\Phi X)^{T}\bigr)
=2\sum_i\la\on_{e_i}(\Phi X),e_i\ra+2\la\Phi X,H\ra
=2\la\Phi X,\Phi\n\theta\ra=2\la X,\n\theta\ra ,
\end{equation*}
where we used \eqref{eqn-LeSh3} and $\eta(\n\theta)=0$. On a Legendrian self-shrinker
$z=4\theta$ by \eqref{eqn-scalar}, and \eqref{eqn-LeM45} follows from \eqref{eqn-Dz} and
\eqref{eqn-drift}. Finally $z=4\theta$ and $H=\Phi\n\theta$ give
$\n z=-4\Phi H$ and $|\n z|^2=16|H|^2$; the last inequality in
\eqref{eqn-height-identities} follows from \eqref{eqn-fineq}.
\end{proof}

\begin{thm}\label{thm-vol}
Let $F:\Sigma^n\to\mathbb R^{2n+1}$ be a complete noncompact Legendrian self-shrinker
on which the function $f$ of \eqref{eqn-f} is proper, and set
$\mathcal S_r:=\{p\in\Sigma^n:\ 2\sqrt{f(p)}\leq r\}$. Then
\begin{equation*}
\int_{\Sigma^n}e^{-f}d\mu<+\infty,\qquad
\Vol(\mathcal S_r)\leq C_1r^n\quad(r\geq1),
\end{equation*}
where $C_1$ depends only on $\int_{\Sigma^n}e^{-f}d\mu$.
These hypotheses hold for every entire Legendrian self-shrinking graph
\eqref{eqn-int-graph}, and for the Legendrian $n$-planes one has
$\Vol(\mathcal S_r)=\omega_nr^n$.
\end{thm}

\begin{proof}
By Lemma \ref{lem-shrinker-identities}, the nonnegative proper function $f$ satisfies
$|\n f|^2\leq f$ and $\D_ff+f=n/2$. The assertions follow from
\cite[Theorem 1.1]{CZ13}.

An entire Legendrian graph is complete by \eqref{eqn-metric}, and
$f=\frac1{16}(|x|^{2}+|Du|^{2})\geq\frac1{16}|x|^{2}$ is proper on it. Finally let $Q$ be a
symmetric $n\times n$ matrix and $c:=-4\sum_i\arctan\lambda_i(Q)$. Then
$u=c+\frac12\la Qx,x\ra$ solves \eqref{eqn-LeSh2}, since $\frac12x\cdot Du-u=-c$, and
$F_u(\mathbb R^n)$ is the Legendrian $n$-plane $\{(x,Qx,c)\}$. Its induced metric is the
constant matrix $\frac14(I+Q^{2})$, and $\mathcal S_r=\{x:\la(I+Q^{2})x,x\ra\leq4r^{2}\}$; in the
coordinates $\frac12(I+Q^{2})^{1/2}x$, which are Euclidean for the induced metric, $\mathcal S_r$
is the ball of radius $r$, so that $\Vol(\mathcal S_r)=\omega_nr^{n}$.
\end{proof}

\vskip24pt

\section{Rigidity under a one-sided bound on the Legendrian angle}\label{sec-rigid}

\begin{thm}\label{thm-general}
Let $F:\Sigma^n\to\mathbb R^{2n+1}$ be a complete properly immersed Legendrian
self-shrinker whose Legendrian angle $\theta$ is bounded from above or from below. Then
$\theta$ is constant, and $F(\Sigma^n)$ is the Legendrian $n$-plane
\begin{equation*}
\bigl\{(w,4\theta):\ w\in\mathcal P\bigr\}\subset\mathbb R^{2n}\times\mathbb R ,
\end{equation*}
where $\mathcal P$ is an $n$-dimensional linear subspace of\/ $\mathbb R^{2n}$ on which
$\sum_idx_i\wedge dy_i$ vanishes.
\end{thm}

\begin{proof}
By \eqref{eqn-scalar} the height $z=4\theta$ is bounded from below or from above on
$\Sigma^n$. Choose $c\in\mathbb R$ with $z\geq c$ on $\Sigma^n$ or $z\leq c$ on $\Sigma^n$,
and put $u:=z-c$ in the first case and $u:=c-z$ in the second. Then $u\geq0$, and
$\mathcal Lu=0$ by \eqref{eqn-height-identities}.

We show that $u$ attains its infimum. For $0<\varepsilon\leq1$ put
$u_\varepsilon:=u+\varepsilon f$. Both summands are nonnegative, so that on a closed
sublevel set $\{u_\varepsilon\leq a\}$ one has $|z-c|\leq a$ and
$\sum_i(x_i^{2}+y_i^{2})=16f\leq16a/\varepsilon$; thus $F$ maps this set into a compact
subset of\/ $\mathbb R^{2n+1}$, and the set is compact because $F$ is proper. Fix
$q_0\in\Sigma^n$. The function $u_\varepsilon$ attains its minimum over the nonempty compact
set $\{u_\varepsilon\leq u_\varepsilon(q_0)\}$, and hence over $\Sigma^n$, at some point
$p_\varepsilon$. There $\n u_\varepsilon=0$ and $\D u_\varepsilon\geq0$, so that by
\eqref{eqn-drift} and \eqref{eqn-fineq}
\begin{equation*}
0\leq\mathcal Lu_\varepsilon(p_\varepsilon)=\varepsilon\,\mathcal Lf(p_\varepsilon)
=\varepsilon\Bigl(\frac n2-f(p_\varepsilon)\Bigr),
\end{equation*}
that is, $f(p_\varepsilon)\leq\frac n2$. Moreover, for every $q\in\Sigma^n$,
\begin{equation}\label{eqn-halfspace-min}
u(p_\varepsilon)\leq u_\varepsilon(p_\varepsilon)\leq u_\varepsilon(q)=u(q)+\varepsilon f(q),
\end{equation}
and in particular $u(p_\varepsilon)\leq u(q_0)+f(q_0)$. Hence all the points
$p_\varepsilon$ lie in the set
\begin{equation*}
K:=\Bigl\{p\in\Sigma^n:\ u(p)\leq u(q_0)+f(q_0),\ f(p)\leq\frac n2\Bigr\},
\end{equation*}
which is compact by the same argument. Choose $\varepsilon_k\downarrow0$ such that
$p_{\varepsilon_k}$ converges to a point $p_*\in K$. Letting $k\to\infty$ in
\eqref{eqn-halfspace-min} gives $u(p_*)\leq u(q)$ for every $q\in\Sigma^n$, so that $u$
attains its minimum at the point $p_*$ of\/ $\Sigma^n$. Since $\mathcal Lu=0$ and $\Sigma^n$ is
connected, the strong maximum principle for $\mathcal L$, applied to $-u$, shows that $u$
is constant. Thus $z$ and $\theta=\frac z4$ are constant. Since $z$ is constant and $F$ is
proper, the function $f=\frac1{16}\sum_i(x_i^{2}+y_i^{2})$ is proper on $\Sigma^n$, and
$\Sigma^n$ is not compact, because $H=\Phi\n\theta=0$ and Lemma \ref{lem1} give
$\D f=\frac n2>0$.

Consequently $H=\Phi\n\theta=0$, the function $z=4\theta$ is constant, and by
\eqref{eqn-Xperp2} the vector field $X^{T}=X-z\,\xi$ is tangent to $\Sigma^n$. By
\eqref{eqn-DX} and \eqref{eqn-nablaxi}, $\on_YX^{T}=Y-z\Phi Y+z\Phi Y=Y$ for
$Y\in T\Sigma^n$; together with $\n f=\frac12X^{T}$ this gives
\begin{equation*}
\Hess f=\frac12\,g ,
\end{equation*}
where $g$ is the induced metric. The function $f$ is not constant, since
$\D f=\frac n2$. By a theorem of Tashiro
\cite[Theorem 2, case (I,B)]{Tas65} the complete manifold $(\Sigma^n,g)$, $n\geq2$, is
therefore isometric to the Euclidean space $\mathbb R^n$. On the other hand, the
tangent planes of\/ $\Sigma^n$ are spanned by orthonormal vectors $e_i\in\mathcal D$ with
$\la\Phi e_i,e_j\ra=0$, so that $\overline K(e_i,e_j)=0$ by \eqref{eqn-curv}, and the Gauss
equation gives for the scalar curvature of\/ $g$, which vanishes,
\begin{equation*}
0=\sum_{i\neq j}\overline K(e_i,e_j)+|H|^{2}-|B|^{2}=-|B|^{2} .
\end{equation*}
Thus $\Sigma^n$ is totally geodesic. (For $n=1$ this is immediate from $B=H=0$.)

Let $p_0$ be a point at which the proper function $f\geq0$ attains its minimum. Then
$X^{T}(p_0)=2\n f(p_0)=0$, and since $X^{T}=X-z\xi=\frac12\sum_i(x_iE_i-y_iE_{n+i})$ this
means $F(p_0)=(0,0,c)$ with $c=4\theta$. As $\Sigma^n$ is complete and totally geodesic,
$F(\Sigma^n)$ is the union of the geodesics of\/ $\mathbb R^{2n+1}$ which start at $F(p_0)$
with initial velocity in $F_*(T_{p_0}\Sigma^n)\subset\mathcal D$. For constants $a_i,b_i$
the curve $\gamma(t)=(2ta,-2tb,c)$ satisfies, by \eqref{eqn-LeM2},
\begin{equation*}
\gamma'=\sum_i\bigl(a_iE_i+b_iE_{n+i}\bigr)-\frac12\sum_i\bigl(a_iy_i+b_ix_i\bigr)\xi
=\sum_i\bigl(a_iE_i+b_iE_{n+i}\bigr),
\end{equation*}
because $a_iy_i+b_ix_i=-2ta_ib_i+2ta_ib_i=0$ along $\gamma$, and then

\begin{equation*}
\on_{\gamma'}\gamma'=\sum_ia_ib_i\bigl(\on_{E_i}E_{n+i}+\on_{E_{n+i}}E_i\bigr)=0
\end{equation*}
 by
\eqref{eqn-covariant}. Hence these straight lines are the geodesics through $(0,0,c)$
tangent to $\mathcal D$, and
$F(\Sigma^n)=\{(w,c):w\in\mathcal P\}$, where

\begin{equation*}
\mathcal P=\{(2a,-2b):\sum_i(a_iE_i+b_iE_{n+i})\in F_*(T_{p_0}\Sigma^n)\} .
\end{equation*}
 Finally
$d\eta=\frac12\sum_idx_i\wedge dy_i$ vanishes on $F_*(T_{p_0}\Sigma^n)$, hence on
$\mathcal P$.
\end{proof}

\begin{cor}\label{cor-slab}
Let $F:\Sigma^n\to\mathbb R^{2n+1}$ be a complete properly immersed Legendrian
self-shrinker.
\begin{itemize}
\item[(i)] If\/ $F(\Sigma^n)$ lies in a half-space $\{z\geq c\}$ or $\{z\leq c\}$, then
$F(\Sigma^n)$ is a Legendrian $n$-plane $\{(w,c'):w\in\mathcal P\}$ as in Theorem
\ref{thm-general}.
\item[(ii)] $\Sigma^n$ is not compact; that is, there are no closed Legendrian
self-shrinkers in $\mathbb R^{2n+1}$.
\end{itemize}
\end{cor}

\begin{proof}
By \eqref{eqn-scalar}, $\theta=\frac z4$ on $\Sigma^n$, so that $\theta\geq\frac c4$ or
$\theta\leq\frac c4$ in case (i), and Theorem \ref{thm-general} applies. If\/ $\Sigma^n$ is compact, then $F$ is proper and
$z$ is bounded, and (i) gives a contradiction, a plane being noncompact.
\end{proof}

\begin{thm}\label{thm-rigid}
If $u$ is an entire smooth solution to the equation \eqref{eqn-LeSh2} in $\mathbb R^n$,
then $u$ is the quadratic polynomial
\begin{equation*}
u(x)=u(0)+\frac{1}{2}\la D^2 u(0)x, x \ra ,\qquad
u(0)=-4\sum_{i=1}^n\arctan\lambda_i\bigl(D^2u(0)\bigr) .
\end{equation*}
\end{thm}

\begin{proof}
By Proposition \ref{pro-pde} the graph $F_u$ is a Legendrian self-shrinker. It is complete
by \eqref{eqn-metric}, it is properly embedded, and $|\theta|<\frac{n\pi}{2}$ by
\eqref{eqn-thetagraph}. By Theorem \ref{thm-general},
$F_u(\mathbb R^n)=\{(w,c):w\in\mathcal P\}$ for an $n$-dimensional linear subspace
$\mathcal P$ of\/ $\mathbb R^{2n}$ and a constant $c$. Since $F_u(\mathbb R^n)$ is a graph
over the $x$-plane, $\mathcal P=\{(x,Qx)\}$ for an $n\times n$ matrix $Q$; hence $Du=Qx$, so
that $Q=D^2u(0)$ is symmetric, and $u-\frac12x\cdot Du=c$. By Lemma \ref{lem-graph},
$u=c+\frac12x\cdot Du=c+\frac12\la Qx,x\ra$ and $c=u(0)$. The value of $u(0)$ is given by
\eqref{eqn-LeSh2} at $x=0$.
\end{proof}

\begin{rem}\label{rem-lagr}
In Theorem \ref{thm-general} and Corollary \ref{cor-slab} the Legendrian setting differs
from the Lagrangian one. For a Lagrangian submanifold of\/ $\mathbb C^n$ the self-shrinker
equation $H=-\frac12P^{\perp}$ does not contain the Lagrangian angle, which is in general
defined only modulo $2\pi$. In particular compactness alone does not imply rigidity: there
are closed Lagrangian self-shrinkers, for instance the circle of radius $\sqrt2$ in
$\mathbb C$, the Clifford torus \cite{CL14} and examples of Anciaux \cite{Anc06}, so that
the Lagrangian analogue of Corollary \ref{cor-slab} (ii) is false. A set of hypotheses
under which the Lagrangian analogue of Theorem \ref{thm-general} does hold is that the
angle be single-valued, which is the vanishing of the Maslov class, and that it be bounded;
both are conditions on the Gauss map rather than on the position of the submanifold. In
Definition \ref{def-shrinker}, a globally defined real branch of the Legendrian angle is
part of the data. The $\xi$-component \eqref{eqn-scalar} of the equation then identifies
$\theta$ with the coordinate function $\frac z4$, so that a bound on it from one side is
equivalent to the submanifold lying in a half-space $\{z\geq c\}$ or $\{z\leq c\}$. The height $c$ of the bounding hyperplane is arbitrary in Corollary \ref{cor-slab} (i).
The Euclidean half-space theorem for properly immersed self-shrinking hypersurfaces
\cite[Theorem 1.1]{CE16} concerns half-spaces bounded by hyperplanes through the origin.
\end{rem}

\begin{pro}\label{pro-fproper}
In Theorem \ref{thm-general} the two hypotheses on $F$ and on $\theta$ can be replaced by
the single hypothesis that the function $f$ of \eqref{eqn-f} be proper: a complete
Legendrian self-shrinker on which $f$ is proper is a Legendrian $n$-plane.
\end{pro}

\begin{proof}
We first show that $\Sigma^n$ is not compact. By \eqref{eqn-LeM45} and \eqref{eqn-drift},
$\Div(e^{-f}\n\theta)=0$. If\/ $\Sigma^n$ were compact, multiplying by $\theta$ and integrating
would give 
\[
\int_{\Sigma^n}|\n\theta|^{2}e^{-f}d\mu=0,
\]
 hence 
 \[
 H=\Phi\n\theta=0
 \]
  and, by
Lemma \ref{lem1}, 
\[
\D f=\frac n2+\frac12\la X,H\ra=\frac n2>0,
\]
 which is impossible on a compact
manifold. Thus Theorem \ref{thm-vol} gives finite weighted volume.

Let $\rho:\mathbb R\to[0,1]$ be a smooth function with $\rho\equiv1$ on $(-\infty,1]$,
$\rho\equiv0$ on $[2,\infty)$ and $|\rho'|\leq2$, and for $r\geq1$ put
$\chi:=\rho\bigl(2\sqrt f/r\bigr)$. Then $\chi\equiv1$ on $\mathcal S_r$, $\chi\equiv0$ outside
$\mathcal S_{2r}$, $\chi$ has compact support because $f$ is proper, and since
$|\n(2\sqrt f)|=|\n f|/\sqrt f\leq1$ on $\{f>0\}$ by \eqref{eqn-fineq},
\begin{equation*}
|\n\chi|\leq\frac2r .
\end{equation*}
The function $\chi$ is smooth also at points where $f=0$, since it is constant near them.
Multiplying
$\Div(e^{-f}\n\theta)=0$ by $\chi^{2}\arctan\theta$ and
integrating by parts gives
\begin{equation*}
\int_{\Sigma^n}\frac{\chi^{2}|\n\theta|^{2}}{1+\theta^{2}}\,e^{-f}d\mu
=-2\int_{\Sigma^n}\chi\arctan\theta\,\la\n\chi,\n\theta\ra\,e^{-f}d\mu .
\end{equation*}
On the support of $\n\chi$ one has $\frac r2\leq\sqrt f\leq r$, and $H=\Phi\n\theta$
gives $|\n\theta|=|H|\leq\sqrt f\leq r$ by \eqref{eqn-fineq}; since
$|\arctan\theta|\leq\frac\pi2$ and $|\n\chi|\leq\frac2r$, the right-hand side is at most
\begin{equation*}
2\cdot\frac\pi2\cdot\frac2r\cdot r\int_{\mathcal S_{2r}\backslash\mathcal S_r}e^{-f}d\mu
=2\pi\int_{\mathcal S_{2r}\backslash\mathcal S_r}e^{-f}d\mu ,
\end{equation*}
which tends to $0$ as $r\to+\infty$ because $\int_{\Sigma^n}e^{-f}d\mu<+\infty$ and the
$\mathcal S_r$ exhaust $\Sigma^n$. Since $\chi\equiv1$ on $\mathcal S_r$, the same bound
holds for $\int_{\mathcal S_r}|\n\theta|^2(1+\theta^2)^{-1}e^{-f}d\mu$.
Letting $r\to+\infty$ gives $\n\theta\equiv0$. Since $f$ is proper by assumption, 
the rest of the proof of Theorem 5.1 applies unchanged.
\end{proof}

\begin{rem}\label{rem-fproper}
The hypotheses of Theorem \ref{thm-general} do not yield the properness of $f$ directly,
since the height need not be bounded a priori; but once its proof has shown that $z$ is
constant, the properness of $F$ gives that of $f$. In this sense Proposition
\ref{pro-fproper}, which assumes neither the properness of $F$ nor a bound on $\theta$, is
the stronger statement. We retain Theorem \ref{thm-general} as the
main result because it is formulated in terms of the properness of the immersion and a
one-sided bound on its height. The horizontal potential $f$ is the auxiliary function of
both arguments. Example \ref{exa-cyl} shows that the properness of $F$ alone does not imply
rigidity.
\end{rem}

\begin{exa}\label{exa-cyl}
Write a point of\/ $\mathbb R^{2n+1}$ as $(w_1,\dots,w_n,z)$ with
$w_j=x_j+\sqrt{-1}\,y_j$, and let
\begin{equation}\label{eqn-cyl}
F_{\mathrm{cyl}}(t,s)=\bigl(2\sqrt2\,e^{\sqrt{-1}\,t},\ s_1,\dots,s_{n-1},\ -4t-2\pi\bigr),
\qquad (t,s)\in\mathbb R\times\mathbb R^{n-1} .
\end{equation}
Then $F_{\mathrm{cyl}}$ is a complete, properly embedded Legendrian self-shrinker in the sense of
Definition \ref{def-shrinker} in $\mathbb R^{2n+1}$ which is not a Legendrian $n$-plane, and its Legendrian angle
$\theta=-t-\frac\pi2$ is unbounded.
\end{exa}

Indeed, the tangent frame of \eqref{eqn-cyl} is
$v_1=(F_{\mathrm{cyl}})_*\frac{\p}{\p t}$, $v_j=(F_{\mathrm{cyl}})_*\frac{\p}{\p s_{j-1}}$ $(2\leq j\leq n)$, and
by \eqref{eqn-LeM2} a vector with complex components $(V_1,\dots,V_n)$ and
$z$-component $V_0$ has $\xi$-component

\begin{equation*}
\frac12\bigl(V_0+\frac12\sum_j\mathrm{Im}(\overline{w_j}V_j)\bigr) ;
\end{equation*}
 along $F_{\mathrm{cyl}}$ this
vanishes for $v_1$, because $\mathrm{Im}(\overline{w_1}\,\sqrt{-1}\,w_1)=|w_1|^{2}=8$ and
$V_0=-4$, so that the $\xi$-component is $\frac12(-4+4)=0$, and for $v_j$, because $w_j$ and $V_j=1$ are real. Hence $F_{\mathrm{cyl}}$ is Legendrian.
Under the identification preceding \eqref{eqn-angledet} the vector $v_j$ corresponds to
$\frac12\overline{V_j}$ in the $j$-th slot, so that the matrix of
$(v_1,\dots,v_n)$ is diagonal with entries
$-\sqrt2\,\sqrt{-1}\,e^{-\sqrt{-1}t}$ and $\frac12$; its determinant is therefore
$\sqrt2\cdot2^{-(n-1)}e^{-\sqrt{-1}(t+\pi/2)}$, and comparing it with
$\sqrt{\det(\la v_i,v_j\ra)}=\sqrt2\cdot2^{-(n-1)}$ gives $\theta=-t-\frac\pi2$ for a
suitable branch, whence
$4\theta=z$ and $F_{\mathrm{cyl}}$ is a Legendrian self-shrinker by Proposition \ref{pro-scalar}. The
induced metric $\la v_1,v_1\ra=2$, $\la v_j,v_j\ra=\frac14$, $\la v_i,v_j\ra=0$ is constant
on $\mathbb R^n$, so $F_{\mathrm{cyl}}$ is complete; $|w_1|\equiv2\sqrt2$ shows that the image is not a Legendrian
$n$-plane; and $F_{\mathrm{cyl}}$ is proper and injective because $s$ is bounded on a bounded set of
values of $(w_2,\dots,w_n)$ and $t$ is determined by $z$.

Thus in Theorem \ref{thm-general} the hypothesis that the Legendrian angle be bounded from
above or from below, that is, by Corollary \ref{cor-slab}, the hypothesis that
$F(\Sigma^n)$ lie in a half-space $\{z\geq c\}$ or $\{z\leq c\}$, cannot be omitted in any dimension, not even for properly embedded submanifolds. Along
\eqref{eqn-cyl} one has
\begin{equation*}
f=\frac1{16}\bigl(8+|s|^{2}\bigr)=\frac12+\frac{|s|^{2}}{16} ,
\end{equation*}
which is independent of $t$; the sequence $(t,s)=(k,0)$, $k\in\mathbb N$, leaves every
compact subset of the domain while $f$ stays equal to $\frac12$, so that $f$ is not proper.
Thus Example \ref{exa-cyl} does not satisfy the hypothesis of Proposition
\ref{pro-fproper}, and it shows that the properness of $F$ cannot replace the properness
of $f$ in that proposition.

\begin{rem}\label{rem-proper}
We do not know whether every complete Legendrian self-shrinker contained in a slab
$\{|z|\leq C\}$, that is, with bounded Legendrian angle, must be a Legendrian $n$-plane
when the immersion is not assumed to be proper. Examples \ref{exa-cyl} and \ref{exa-AL} do not answer this question, since both have
unbounded Legendrian angle. By contrast, the properness of $f$ cannot simply be dropped in
Proposition \ref{pro-fproper}, as Example \ref{exa-cyl} shows.
\end{rem}

\begin{rem}\label{rem-twosided}
If the Legendrian angle is bounded, the integral method of Ding and Xin \cite{DX14} gives a
second proof that $\theta$ is constant in Theorem \ref{thm-general}. In this case
$z=4\theta$ is bounded by \eqref{eqn-scalar}, so that $f$ is proper because $F$ is, and the
proof of Proposition \ref{pro-fproper} shows that $\Sigma^n$ is not compact and that Theorem
\ref{thm-vol} applies. With the cut-off function $\chi$ of that proof, multiplying
$\Div(e^{-f}\n\theta)=0$ by
$\theta\chi^{2}$ and integrating by parts (cf.\ \cite{DX14}),
\begin{equation}\label{eqn-LeM46}\aligned
0 =&\ \int_{\Sigma^n} \theta \chi^2\, \Div(e^{-f}\n \theta)\,d\mu
= - \int_{\Sigma^n} \la \n(\theta\chi^2), \n \theta \ra e^{-f}d\mu \\
=&\ - \int_{\Sigma^n} |\n \theta|^2 \chi^2 e^{-f}d\mu
- 2\int_{\Sigma^n} \la \n\chi, \n\theta \ra \theta\chi\, e^{-f}d\mu ,
\endaligned\end{equation}
and by the Cauchy--Schwarz inequality
\begin{equation}\label{eqn-LeM47}
- 2\int_{\Sigma^n} \la \n\chi, \n\theta \ra \theta\chi\, e^{-f}d\mu
\leq \frac{1}{2}\int_{\Sigma^n} |\n\theta|^2 \chi^2 e^{-f}d\mu
+ 2\int_{\Sigma^n} \theta^2|\n\chi|^2 e^{-f}d\mu .
\end{equation}
From \eqref{eqn-LeM46} and \eqref{eqn-LeM47},
\begin{equation}\label{eqn-LeM48}
\int_{\Sigma^n} |\n\theta|^2 \chi^2 e^{-f}d\mu
\leq 4\int_{\Sigma^n} \theta^2|\n\chi|^2 e^{-f}d\mu .
\end{equation}
Since $\n\chi$ is supported in $\mathcal S_{2r}\backslash \mathcal S_r$, where $f\geq\frac{r^{2}}4$, Theorem
\ref{thm-vol} gives
\begin{equation}\label{eqn-LeM49}\aligned
\int_{\mathcal S_r} |\n\theta|^2 e^{-f}d\mu
\leq &\ \frac{16}{r^2}\,\sup_{\Sigma^n}\theta^{2}\int_{\mathcal S_{2r}\backslash \mathcal S_r} e^{-f}d\mu
\leq \frac{16}{r^2}\,\sup_{\Sigma^n}\theta^{2}\; e^{-\frac{r^2}{4}}\,\Vol(\mathcal S_{2r}) \\
\leq &\ 2^{n+4}\, C_1 \sup_{\Sigma^n}\theta^{2}\; r^{n-2}e^{-\frac{r^2}{4}} .
\endaligned\end{equation}
Letting $r\to +\infty$ we conclude that $\n\theta\equiv0$, so that $\theta$ is constant.
For this purpose only the finiteness of the weighted
volume is needed: the right-hand side of \eqref{eqn-LeM48} is at most
$\frac{16}{r^{2}}\sup\theta^{2}\int_{\Sigma^n}e^{-f}d\mu$. The volume growth of Theorem
\ref{thm-vol} gives in addition the rate in \eqref{eqn-LeM49}. This argument uses the
boundedness of $\theta$ on both sides and does not cover the one-sided hypothesis of
Theorem \ref{thm-general}.
\end{rem}

\vskip24pt

\section{A sharp cylinder theorem}\label{sec-cylinder}

Example \ref{exa-cyl} shows that the properness of $F$ does not imply the properness of
$f$: along $F_{\mathrm{cyl}}$ the height is unbounded while $|w|$ stays bounded. In this
section we study the horizontal position of a properly immersed Legendrian self-shrinker
without any assumption on the height. As in the Introduction, write $w=x+\sqrt{-1}\,y$ and
\begin{equation*}
\mathcal C_n=\{(w,z):|w|^2=8n\} .
\end{equation*}
Since $f=\frac1{16}|w|^2$ along the immersion by \eqref{eqn-f},
$F(\Sigma^n)\subset\mathcal C_n$ is equivalent to $f\equiv\frac n2$.

\subsection{The maximum principle and cylinder rigidity}

We use the following elementary form of the exhaustion-function criterion for
parabolicity, often called Khas'minskii's criterion; see \cite{Gri99} for its
potential-theoretic background. For self-shrinkers of the mean curvature flow in Euclidean
space, this potential-theoretic approach to confinement theorems goes back to Pigola and
Rimoldi \cite{PR14}; see also \cite{IPR21}, where the cylinder theorems of \cite{CE16} are
recovered by it. The proof below is written directly for the drift operator
\eqref{eqn-drift}.

\begin{lem}\label{lem-exhaustion}
Suppose that a nonnegative proper function $\sigma\in C^2(\Sigma^n)$ satisfies
$\mathcal L\sigma\leq0$ outside a compact set. Then every function $v\in C^2(\Sigma^n)$
which is bounded above and satisfies $\mathcal Lv\geq0$ is constant.
\end{lem}

\begin{proof}
If\/ $\Sigma^n$ is compact, the assertion is the strong maximum principle. Otherwise choose a
relatively compact domain $\Omega$ whose interior contains the compact set outside of which
$\mathcal L\sigma\leq0$. For $\varepsilon>0$ the function $v-\varepsilon\sigma$ tends to
$-\infty$ at infinity, because $v$ is bounded above and $\sigma$ is proper, and it satisfies
$\mathcal L(v-\varepsilon\sigma)\geq0$ outside $\overline\Omega$. Hence
\begin{equation*}
v(p)-\varepsilon\sigma(p)
\leq\max_{\overline\Omega}(v-\varepsilon\sigma)
\leq\max_{\overline\Omega}v
\qquad(p\in\Sigma^n):
\end{equation*}
indeed, the open set on which the first inequality fails has compact closure, disjoint
from $\overline\Omega$, so that on each of its components $v-\varepsilon\sigma$ would
attain an interior maximum, contradicting the strong maximum principle. Letting
$\varepsilon\downarrow0$ shows that $v$ attains its maximum over $\Sigma^n$ at a point of\/
$\overline\Omega$, and the strong maximum principle on $\Sigma^n$ shows that $v$ is
constant.
\end{proof}

\begin{thm}\label{thm-cylinder}
Let $F:\Sigma^n\to\mathbb R^{2n+1}$ be a properly immersed Legendrian self-shrinker.
If either $f\leq\frac n2$ everywhere or $f\geq\frac n2$ everywhere, then
\begin{equation*}
f\equiv\frac n2,\qquad X^T\equiv0,\qquad |H|^2\equiv\frac n2 .
\end{equation*}
Thus every such immersion either lies in $\mathcal C_n$ or meets both
$\{|w|^2<8n\}$ and $\{|w|^2>8n\}$.
\end{thm}

\begin{proof}
Since $\mathcal Lz=0$ by \eqref{eqn-height-identities}, the chain rule gives
\begin{equation}\label{eqn-log-barrier}
\mathcal L\log(1+z^2)
=\frac{2(1-z^2)}{(1+z^2)^2}\,|\n z|^2 .
\end{equation}

First suppose that $f\leq\frac n2$. Then $f$ is bounded above, and
$\mathcal Lf=\frac n2-f\geq0$ by \eqref{eqn-fineq}. The function $\sigma=\log(1+z^2)$
is proper on $\Sigma^n$: on a sublevel set of $\sigma$ the height $z$ is bounded, $|w|$ is
bounded on all of $F(\Sigma^n)$ by assumption, and the sublevel set is therefore compact
because $F$ is proper. By \eqref{eqn-log-barrier}, $\mathcal L\sigma\leq0$ outside the
compact set $\{|z|\leq1\}$. Hence $f$ is constant by Lemma \ref{lem-exhaustion}, and
$\mathcal Lf=0$ gives $f=\frac n2$.

Now suppose that $f\geq\frac n2$. The function $-f$ is bounded above and satisfies
$\mathcal L(-f)=f-\frac n2\geq0$. Put
\begin{equation*}
\sigma=f+\frac1{64}\log(1+z^2) .
\end{equation*}
Both summands are nonnegative, so that on a sublevel set of $\sigma$ both $|w|$ and $z$
are bounded, and the sublevel set is compact because $F$ is proper. Where $|z|\geq1$,
\eqref{eqn-log-barrier} gives $\mathcal L\sigma\leq\frac n2-f\leq0$. Where $|z|\leq1$,
\eqref{eqn-fineq}, \eqref{eqn-height-identities} and \eqref{eqn-log-barrier} give
\begin{equation*}
\mathcal L\sigma\leq\frac n2-f+\frac1{32}|\n z|^2
\leq\frac n2-\frac f2 .
\end{equation*}
Consequently $\mathcal L\sigma\leq0$ outside the compact set
$\{|z|\leq1,\ f\leq n\}$, and Lemma \ref{lem-exhaustion}, applied now to $-f$, shows
that $f$ is constant, hence equal to $\frac n2$. Finally $X^T=2\n f=0$ by
\eqref{eqn-Lem1}, and $|H|^2=f-|\n f|^2=\frac n2$ by \eqref{eqn-fineq}.
\end{proof}

\begin{cor}
If a properly immersed Legendrian self-shrinker satisfies $|H|^2\geq\frac n2$
everywhere, then $f\equiv|H|^2\equiv\frac n2$ and its image lies in $\mathcal C_n$.
\end{cor}

\begin{proof}
By \eqref{eqn-fineq}, $f=|\n f|^2+|H|^2\geq\frac n2$, so that Theorem
\ref{thm-cylinder} applies.
\end{proof}

\begin{cor}
Let $F:\Sigma^n\to\mathbb R^{2n+1}$ be a properly immersed Legendrian self-shrinker on
which $|w|$ is bounded. Then every function on $\Sigma^n$ which is bounded above and
$\mathcal L$-subharmonic is constant, and $z(\Sigma^n)=\mathbb R$; in particular the
Legendrian angle is unbounded in both directions.
\end{cor}

\begin{proof}
As in the proof of Theorem \ref{thm-cylinder}, $\log(1+z^2)$ is a nonnegative proper
function with $\mathcal L\log(1+z^2)\leq0$ outside the compact set $\{|z|\leq1\}$, so
that Lemma \ref{lem-exhaustion} applies. If\/ $z$ were bounded above or bounded below,
the lemma, applied to $z$ or to $-z$, would show that $z$ is constant, since
$\mathcal Lz=0$. Then $F(\Sigma^n)$ would lie in a compact set, $\Sigma^n$ would be
compact because $F$ is proper, and this contradicts Corollary \ref{cor-slab} (ii). Thus
$z$ is unbounded in both directions, and $z(\Sigma^n)$, being an interval by
connectedness, is $\mathbb R$.
\end{proof}

\begin{rem}
As in Remark \ref{rem-proper}, we do not know whether the properness of $F$ in Theorem
\ref{thm-cylinder} can be replaced by the completeness of\/ $\Sigma^n$. The exhaustion
functions used in the proof are proper only because $F$ is.
\end{rem}

\subsection{The geometry of the equality case}

Set $R_0:=\sqrt{8n}$ and, for $s\in\mathbb R$,
\begin{equation}\label{eqn-screw-flow}
\mathcal R_s(w,z)=\bigl(e^{-\sqrt{-1}s}w,\ z+4ns\bigr),\qquad
V=\sum_{j=1}^n\Bigl(y_j\frac{\p}{\p x_j}-x_j\frac{\p}{\p y_j}\Bigr)
+4n\frac{\p}{\p z} .
\end{equation}
The maps $\mathcal R_s$ preserve $\eta$ and $\overline g$, and $V$ is their generator.
By \eqref{eqn-LeM2},
\begin{equation}\label{eqn-screw-contact}
V=\frac12\sum_{j=1}^n\bigl(y_jE_j+x_jE_{n+j}\bigr)+\Bigl(2n-\frac14|w|^2\Bigr)\xi ,
\qquad \eta(V)=2n-\frac14|w|^2 ,
\end{equation}
so that, comparing with $\Phi X=\frac12\sum_j(x_jE_{n+j}+y_jE_j)$ from the proof of
Lemma \ref{lem-DX}, one has $V=\Phi X^{\mathcal H}$ and $|V|^2=2n$ on $\mathcal C_n$.

\begin{lem}\label{lem-screw}
Let $n\geq1$ and let $F:\Sigma^n\to\mathbb R^{2n+1}$ be a properly immersed Legendrian
self-shrinker with $f\equiv\frac n2$. Then the tangent vector field $Y:=\frac12\n z$
satisfies
\begin{equation}\label{eqn-screw-field}
F_*Y=V,\qquad |Y|^2=2n,\qquad Y(z)=4n ,
\end{equation}
its flow $\mathcal U_s$ is defined for all $s\in\mathbb R$, and
\begin{equation}\label{eqn-screw-equivariance}
F\circ\mathcal U_s=\mathcal R_s\circ F,\qquad z\circ\mathcal U_s=z+4ns .
\end{equation}
\end{lem}

\begin{proof}
Since $f$ is constant, \eqref{eqn-Lem1} gives $X^T=0$, so that $X=X^\perp=-2H+z\xi$ by
\eqref{eqn-Xperp2} and $X^{\mathcal H}=-2H$. By \eqref{eqn-LeSh3} and
\eqref{eqn-height-identities}, $\Phi X^{\mathcal H}=-2\Phi H=2\n\theta=\frac12\n z$,
which is $V$ along $F$ by \eqref{eqn-screw-contact}, since $F(\Sigma^n)\subset\mathcal C_n$.
The remaining identities in \eqref{eqn-screw-field} follow from $|\n z|^2=16|H|^2=8n$,
which is \eqref{eqn-height-identities} together with $|H|^2=\frac n2$ (Theorem
\ref{thm-cylinder}). The induced metric is complete (Section \ref{sec-prelim}) and $Y$
has constant length, so that its flow is defined for all $s$. Since $F_*Y=V$ and $V$
generates $\{\mathcal R_s\}$, uniqueness of integral curves gives
$F\circ\mathcal U_s=\mathcal R_s\circ F$, and the second identity in
\eqref{eqn-screw-equivariance} is the last component of the first.
\end{proof}

We use the standard contact structure on the round sphere $S^{2n-1}(R_0)$ of the
Euclidean metric $g_{\mathrm{Euc}}=\sum_i(dx_i^2+dy_i^2)$ of\/ $\mathbb C^n$, whose contact
distribution is $\ker\mathrm{Im}\sum_j\overline w_j\,dw_j$; a Legendrian submanifold of
$S^{2n-1}(R_0)$ has dimension $n-1$, and the cone over it is a Lagrangian submanifold of\/
$\mathbb C^n$. Let $\iota:\Gamma^{n-1}\to S^{2n-1}(R_0)$ be a Legendrian immersion of an
oriented manifold. For a positively oriented frame $v_1,\dots,v_{n-1}$ of $T_q\Gamma$ the
vectors $d\iota(v_1),\dots,d\iota(v_{n-1}),\iota(q)$ span the tangent space of the cone,
which is a Lagrangian subspace of\/ $\mathbb C^n$; hence they are linearly independent over
$\mathbb C$, and the determinant ${\det}_{\mathbb C}(d\iota(v_1),\dots,d\iota(v_{n-1}),\iota)$,
computed in the standard coordinates of\/ $\mathbb C^n$ and with the radial vector in the
last slot, is nonzero. Its argument does not depend on the choice of the frame, because a
change of positively oriented frames multiplies the determinant by a positive number. The
immersion $\iota$ is minimal in the sphere if and only if this argument is constant; this
is the constant-phase characterisation of minimal Legendrian submanifolds of the sphere
\cite{Has04}.

\begin{thm}\label{thm-cylinder-structure}
Let $n\geq2$ and let $F:\Sigma^n\to\mathbb R^{2n+1}$ be a properly immersed
Legendrian self-shrinker with $f\equiv\frac n2$. Then $\Gamma:=z^{-1}(0)$ is a compact
connected hypersurface of\/ $\Sigma^n$, and with the flow $\mathcal U_s$ of Lemma
\ref{lem-screw} the map $\Gamma\times\mathbb R\to\Sigma^n$, $(q,s)\mapsto\mathcal U_s(q)$,
is a diffeomorphism, in terms of which
\begin{equation}\label{eqn-screw-model}
F(q,s)=\bigl(e^{-\sqrt{-1}s}\iota(q),\ 4ns\bigr),\qquad
\theta(q,s)=ns ,
\end{equation}
where $\iota:\Gamma\to S^{2n-1}(R_0)$, defined by $F|_\Gamma=(\iota,0)$, is a minimal
Legendrian immersion. Orient $\Gamma$ so that a positive frame of\/ $\Gamma$ followed by
$\n z$ is a positive frame of\/ $\Sigma^n$. Then
\begin{equation}\label{eqn-slice-phase}
{\det}_{\mathbb C}\bigl(d\iota(v_1),\dots,d\iota(v_{n-1}),\iota\bigr)
\in\sqrt{-1}\,(0,\infty)
\end{equation}
for every positively oriented frame $v_1,\dots,v_{n-1}$ of\/ $\Gamma$, and the induced
metric is the product
\begin{equation}\label{eqn-screw-metric}
g=g_\Gamma+2n\,ds^2,\qquad
g_\Gamma=\frac14\,\iota^*g_{\mathrm{Euc}} .
\end{equation}
Conversely, if\/ $\Gamma^{n-1}$ is a compact connected oriented manifold and
$\iota:\Gamma\to S^{2n-1}(R_0)$ is a Legendrian immersion satisfying
\eqref{eqn-slice-phase}, then \eqref{eqn-screw-model} defines a complete properly immersed
Legendrian self-shrinker $F:\Gamma\times\mathbb R\to\mathbb R^{2n+1}$ with $f\equiv\frac n2$
and induced metric \eqref{eqn-screw-metric}, which is an embedding if $\iota$ is.
\end{thm}

\begin{proof}
By Lemma \ref{lem-screw}, $Y(z)=4n$, so that $0$ is a regular value of $z$ and $\Gamma$
is a hypersurface, and by \eqref{eqn-screw-equivariance} every integral curve of $Y$ meets
$\Gamma$ exactly once. Since $F(\Sigma^n)\subset\mathcal C_n$, $\Gamma$ is the inverse image
under the proper map $F$ of the compact set $S^{2n-1}(R_0)\times\{0\}$, hence compact. The
map $(q,s)\mapsto\mathcal U_s(q)$ is smooth with smooth inverse
\begin{equation*}
p\longmapsto\Bigl(\mathcal U_{-z(p)/(4n)}(p),\ \frac{z(p)}{4n}\Bigr) ,
\end{equation*}
hence a diffeomorphism, and $\Gamma$ is connected because $\Sigma^n$ is. Writing
$F(q)=(\iota(q),0)$ for $q\in\Gamma$ and using \eqref{eqn-screw-equivariance} we obtain
the first formula in \eqref{eqn-screw-model}; the second is $4\theta=z$.

Since $dz=0$ on $T\Gamma$ and $F$ is an immersion, $\iota$ is an immersion, and
$F^*\eta=0$ restricted to $T\Gamma$ reads $\mathrm{Im}\sum_j\overline{\iota_j}\,d\iota_j=0$;
thus $\iota$ is Legendrian in $S^{2n-1}(R_0)$, its dimension being $n-1$. Let
$v_1,\dots,v_{n-1}$ be a positively oriented frame of\/ $\Gamma$ at $q$; then
$(v_1,\dots,v_{n-1},Y)$ is a positively oriented frame of\/ $\Sigma^n$. By the frame
convention preceding \eqref{eqn-angledet}, a horizontal vector whose complex coordinate
component is $W$ is represented by $\frac12\overline W$, and by \eqref{eqn-screw-flow} the
complex component of $F_*Y=V$ at $(\iota(q),0)$ is $-\sqrt{-1}\,\iota(q)$. Hence
\begin{equation}\label{eqn-slice-determinant}
{\det}_{\mathbb C}\bigl(F_*v_1,\dots,F_*v_{n-1},F_*Y\bigr)
=2^{-n}\sqrt{-1}\;
\overline{{\det}_{\mathbb C}\bigl(d\iota(v_1),\dots,d\iota(v_{n-1}),\iota\bigr)} ,
\end{equation}
where the determinant on the left is the one of \eqref{eqn-angledet} and the one on the
right is computed in the standard coordinates of\/ $\mathbb C^n$. Since $\theta=0$ on
$\Gamma$, the left-hand side is positive by \eqref{eqn-angledet}, which proves
\eqref{eqn-slice-phase}; and the constancy of the phase shows that $\iota$ is minimal.

For the converse, let $F$ be defined by \eqref{eqn-screw-model}. From \eqref{eqn-LeM1},
\begin{equation*}
F^*\eta=\Bigl(2n-\frac14|\iota|^2\Bigr)ds
+\frac14\,\mathrm{Im}\sum_j\overline{\iota_j}\,d\iota_j=0 ,
\end{equation*}
because $|\iota|^2=8n$ and $\iota$ is Legendrian. The vector $F_*\frac{\p}{\p s}$ has
height component $4n\neq0$, while $F_*(T\Gamma)$ lies in $\{dz=0\}$; as $\iota$ is an
immersion, $F$ is an immersion. At the parameter $s$ the complex components of $F_*v_k$
and of $F_*\frac{\p}{\p s}$ are $e^{-\sqrt{-1}s}$ times those at $s=0$, so that the
determinant \eqref{eqn-slice-determinant} is multiplied by $e^{\sqrt{-1}ns}$; by
\eqref{eqn-slice-phase} it is positive at $s=0$, and \eqref{eqn-angledet} gives the
global real branch $\theta=ns$ of the Legendrian angle. Since $z=4ns$, $F$ is a
Legendrian self-shrinker by Proposition \ref{pro-scalar}, and
$f=\frac1{16}|\iota|^2=\frac n2$.

In both directions the induced metric is computed from
$\overline g=\frac14\sum_i(dx_i^2+dy_i^2)+\eta\otimes\eta$ and $F^*\eta=0$: on $T\Gamma$
it is $\frac14\iota^*g_{\mathrm{Euc}}$, the squared length of $F_*\frac{\p}{\p s}$ is
$\frac14|\iota|^2=2n$, and the inner product of $F_*\frac{\p}{\p s}$ with $F_*v$,
$v\in T\Gamma$, is $-\frac14\mathrm{Im}\sum_j\overline{\iota_j}\,d\iota_j(v)=0$. This
proves \eqref{eqn-screw-metric}. The product of the compact manifold $(\Gamma,g_\Gamma)$
with a line is complete. Since the height determines $s$ and $\Gamma$ is compact, $F$ is
proper, and it is injective if $\iota$ is.
\end{proof}

\begin{rem}
In the equality case, $e^{-f}=e^{-n/2}$ is constant. The product metric
\eqref{eqn-screw-metric} therefore has infinite weighted volume, and $f$ is not
proper. These examples lie outside the hypotheses of Theorem \ref{thm-vol}, whereas
the exhaustion argument of Theorem \ref{thm-cylinder} applies to them.
\end{rem}

\begin{exa}\label{exa-round-slice}
Let $n\geq2$, and give the unit sphere $S^{n-1}\subset\mathbb R^n$ the orientation for
which the determinant $\det(a_1,\dots,a_{n-1},a)$ is positive whenever $a_1,\dots,a_{n-1}$
is a positively oriented frame of\/ $T_aS^{n-1}$. Then
\begin{equation}\label{eqn-round-slice}
F_{\mathrm{sph}}(a,s)=\Bigl(\sqrt{8n}\,e^{\sqrt{-1}(\pi/(2n)-s)}\,a,\ 4ns\Bigr),
\qquad (a,s)\in S^{n-1}\times\mathbb R,
\end{equation}
is a complete properly embedded Legendrian self-shrinker lying in $\mathcal C_n$, with
Legendrian angle $\theta=ns$ and induced metric $2n\,(g_{S^{n-1}}+ds^2)$. Indeed, its
zero-height slice $\iota(a)=\sqrt{8n}\,e^{\sqrt{-1}\pi/(2n)}a$ is a totally geodesic
sphere of\/ $S^{2n-1}(R_0)$, Legendrian because $a\cdot da=0$, and
\begin{equation*}
{\det}_{\mathbb C}\bigl(d\iota(a_1),\dots,d\iota(a_{n-1}),\iota\bigr)
=(8n)^{n/2}e^{\sqrt{-1}\pi/2}\det(a_1,\dots,a_{n-1},a)\in\sqrt{-1}\,(0,\infty) ,
\end{equation*}
so that Theorem \ref{thm-cylinder-structure} applies.
\end{exa}

\begin{exa}\label{exa-torus}
For $n=2$ the Clifford model, which is the model (2) of \cite[Theorem 1.3]{CWZ24}, reads
in the normalisation \eqref{eqn-scalar}
\begin{equation*}
F_{\mathrm{tor}}(t_1,t_2)=\bigl(2\sqrt2\,e^{\sqrt{-1}t_1},\ 2\sqrt2\,e^{\sqrt{-1}t_2},\
-4(t_1+t_2)+4\pi\bigr),\qquad (t_1,t_2)\in\mathbb R^{2} .
\end{equation*}
It has Legendrian angle $\theta=\pi-t_1-t_2$ and lies in $\mathcal C_2$, and it is a
complete Legendrian self-shrinker which is \emph{not} properly immersed: since
$F_{\mathrm{tor}}(t_1+2\pi k,t_2-2\pi k)=F_{\mathrm{tor}}(t_1,t_2)$ for every
$k\in\mathbb Z$, the preimage of a point contains an unbounded sequence in $\mathbb R^2$.
The map $F_{\mathrm{tor}}$ descends to the quotient of\/ $\mathbb R^{2}$ by
$(t_1,t_2)\mapsto(t_1+2\pi,t_2-2\pi)$, a cylinder, on which the same formula is a proper
embedding; it is thus the parametrisation on $\mathbb R^{2}$, and not the image, that
fails to be proper. The zero-height slice of the quotient is the Legendrian great circle
$\{2\sqrt2\,(e^{\sqrt{-1}t},-e^{-\sqrt{-1}t}):t\in\mathbb R\}$ of\/ $S^{3}(4)$, which
satisfies \eqref{eqn-slice-phase} for the orientation of decreasing $t$, and
$F_{\mathrm{tor}}$ is its screw sweep \eqref{eqn-screw-model}.
\end{exa}

\begin{cor}\label{cor-lowdim}
\begin{itemize}
\item[(i)] For $n=1$, a properly immersed Legendrian self-shrinker lying in
$\mathcal C_1$ has, up to reparametrisation, one of the two images
\begin{equation*}
t\longmapsto\bigl(2\sqrt2\,e^{\sqrt{-1}t},-4t-2\pi\bigr),\qquad
t\longmapsto\bigl(2\sqrt2\,e^{\sqrt{-1}t},-4t+2\pi\bigr),
\end{equation*}
with the orientations and angle branches required by $4\theta=z$. The two images are
interchanged by the rotation $(w,z)\mapsto(-w,z)$, and the first is the curve
$F_{\mathrm{cyl}}$ of Example \ref{exa-cyl} for $n=1$.
\item[(ii)] For $n=2$, every properly immersed Legendrian self-shrinker lying in
$\mathcal C_2$ is, up to a transformation in $SU(2)$ and reparametrisation, a finite
cover of the proper Clifford model of Example \ref{exa-torus}.
\end{itemize}
\end{cor}

\begin{proof}
(i) By Lemma \ref{lem-screw}, which does not require $n\geq2$, the vector field $Y$ has
constant length $\sqrt2$ and is complete; since $\Sigma^1$ is connected, a single integral
curve of $Y$ covers it, and by \eqref{eqn-screw-equivariance} its image is a helix
$w=2\sqrt2\,e^{\sqrt{-1}t}$, $z=-4t+c$. For the orientation of increasing $t$,
\eqref{eqn-angledet} gives $\theta=-t-\frac\pi2+2\pi k$, as in Example \ref{exa-cyl}, so
that $4\theta=z$ forces $c=-2\pi+8\pi k$; the opposite orientation adds $\pi$ to $\theta$
and gives $c=2\pi+8\pi k$. Replacing $t$ by $t+2\pi k$ identifies the images with the two
displayed ones, and the rotation $w\mapsto-w$, that is, $t\mapsto t+\pi$, interchanges
them. Since the height is strictly monotone along the helix, the integral curve is
injective.

(ii) By Theorem \ref{thm-cylinder-structure}, $\Gamma$ is a compact connected minimal
Legendrian curve in $S^3(4)$, that is, a Legendrian closed geodesic of the round sphere,
hence a finite cover of a Legendrian great circle. Two oriented Legendrian great circles
$\iota$ satisfying \eqref{eqn-slice-phase} are related by an element of $SU(2)$: the
matrices $\bigl(\iota'/|\iota'|,\iota/4\bigr)$ are unitary, because $\iota'$ is
orthogonal to $\iota$ and to $\sqrt{-1}\,\iota$, and \eqref{eqn-slice-phase} fixes their
determinant to be $\sqrt{-1}$. The great circle of Example \ref{exa-torus} satisfies
\eqref{eqn-slice-phase}, and the screw sweep \eqref{eqn-screw-model} is determined by the
slice; this proves the assertion.
\end{proof}

\begin{rem}
For $n=3$, if the slice $\Gamma$ of Theorem \ref{thm-cylinder-structure} is a two-sphere,
then $\iota$ is totally geodesic by the theorem of Haskins \cite{Has04} on minimal
Legendrian two-spheres in $S^5$, and the phase normalisation \eqref{eqn-slice-phase}
identifies $F$ with the immersion \eqref{eqn-round-slice} of Example \ref{exa-round-slice}, up to a
transformation in $SU(3)$ and reparametrisation.
\end{rem}

\subsection{Examples and sharpness}\label{subsec-cylinder-examples}

\begin{exa}\label{exa-AL}
For $n=1$ there are examples which are not orbits of any one-parameter group of isometries
of\/ $(\mathbb R^{3},\overline g)$. Let $x,y,\alpha:\mathbb R\to\mathbb R$ solve
\begin{equation}\label{eqn-AL}
\frac{dx}{dt}=xy,\qquad \frac{dy}{dt}=8-x^{2},\qquad \frac{d\alpha}{dt}=x ,
\end{equation}
with $x(0)>0$ and $(x(0),y(0))\neq(2\sqrt2,0)$, and put
\begin{equation*}
F_{\mathrm{AL}}(t)=\bigl(y(t)\cos\alpha(t)+x(t)\sin\alpha(t),\ x(t)\cos\alpha(t)-y(t)\sin\alpha(t),\
4\,\alpha(t)\bigr) .
\end{equation*}
Then $F_{\mathrm{AL}}$ is a complete, properly embedded Legendrian self-shrinker in $\mathbb R^{3}$
with Legendrian angle $\alpha$ and non-constant curvature $\frac x4>0$, and $\alpha$ is
unbounded. This is the Legendrian Abresch--Langer curve of \cite[Theorem 1.2]{CWZ24}, where
the coefficient $8$ in \eqref{eqn-AL} is left as a free parameter and the constant on the
right-hand side of \eqref{eqn-LeSh} is $-4$ divided by it; the value $8$ is the one for
which that constant is $-\frac12$, and the curves obtained from other values are the images
of these under contact dilations $\d_\lambda$.
\end{exa}

Here $\eta(\dot F_{\mathrm{AL}})=0$ and $4\alpha=z$, which is \eqref{eqn-scalar}. The system
\eqref{eqn-AL} has the first integral $E=x^{2}+y^{2}-16\log x$, which is proper on the
half-plane $\{x>0\}$ and attains its minimum exactly at the equilibrium $(2\sqrt2,0)$;
hence the level set through a point other than the equilibrium is a closed curve around it,
$(x,y)$ is periodic and non-constant, $x$ takes its values in a compact subinterval
$[x_-,x_+]$ of\/ $(0,\infty)$, and the solution is defined for all $t$. From
$\|\dot F_{\mathrm{AL}}\|_{\overline g}=4$ the curve is complete, and from
$\frac{dz}{dt}=4x\geq4x_->0$ the coordinate $z$ is a strictly increasing proper function of
$t$, so that $F_{\mathrm{AL}}$ is injective and proper and $\alpha=\frac z4$ is unbounded. Since $x$
is not constant, neither is the curvature, and therefore $F_{\mathrm{AL}}$ is not an orbit of a
one-parameter group of isometries of\/ $(\mathbb R^{3},\overline g)$. For the excluded
equilibrium solution, setting the parameter in \eqref{eqn-cyl} equal to
$-\alpha(t)-\frac\pi2$ and applying the rotation
$(x_1,y_1,z)\mapsto(-x_1,-y_1,z)$ gives $F_{\mathrm{AL}}(t)$; this curve is such an orbit.

\begin{rem}
Examples \ref{exa-cyl} and \ref{exa-AL} are two members of a longer list of complete,
properly embedded Legendrian self-shrinkers with unbounded Legendrian angle: for $n=2$ the
models (3) and (4) of \cite[Theorem 1.3]{CWZ24}, with the constants there determined by
\eqref{eqn-scalar}, are further such examples in $\mathbb R^{5}$, as is the Clifford
model of Example \ref{exa-torus} on its quotient cylinder.
\end{rem}

\begin{pro}\label{pro-sharp-curve}
For $n=1$, the threshold in Theorem \ref{thm-cylinder} is sharp on both sides.
\end{pro}

\begin{proof}
For the non-equilibrium curves of Example \ref{exa-AL} one has $f=\frac1{16}(x^2+y^2)$.
The first integral $E=x^2+y^2-16\log x$ has its unique minimum at $(2\sqrt2,0)$, and on
the axis $y=0$ the function $x^2-16\log x$ is strictly convex, so that every level of\/ $E$
above the minimum meets this axis at two points $x_-<2\sqrt2<x_+$. Hence $f$ takes values
on both sides of\/ $\frac12$ along every such curve. As the level tends to its minimum, the
closed orbit contracts to the equilibrium and $\sup|f-\frac12|$ tends to zero. Since these
curves are properly embedded by Example \ref{exa-AL}, the threshold $8$ of Theorem
\ref{thm-cylinder} can be neither increased on the inner side nor decreased on the outer
side.
\end{proof}

We next give the higher-dimensional sharpness examples. Radial-profile constructions
with spherical Legendrian factors occur in Anciaux \cite{Anc06}. The following
family is verified directly in the present contact and angle conventions.

\begin{pro}\label{pro-sharp-radial}
For every $n\geq2$ and $\varepsilon>0$ there is a complete properly embedded
Legendrian self-shrinker $F_{\mathrm{rad}}:S^{n-1}\times\mathbb R\to\mathbb R^{2n+1}$
such that
\begin{equation*}
\inf f<\frac n2<\sup f,\qquad
\sup\Bigl|f-\frac n2\Bigr|<\varepsilon .
\end{equation*}
Consequently the threshold in Theorem \ref{thm-cylinder} is sharp in every dimension.
\end{pro}

\begin{proof}
Let $r,\vartheta,\beta$ solve
\begin{equation}\label{eqn-radial-system}
\dot r=\cos\vartheta,\qquad
\dot\vartheta=\Bigl(\frac r8-\frac nr\Bigr)\sin\vartheta,\qquad
\dot\beta=\frac{\sin\vartheta}{r},
\end{equation}
with $r>0$ and $0<\vartheta<\pi$, and set
\begin{equation}\label{eqn-radial-model}
F_{\mathrm{rad}}(a,t)=\Bigl(r(t)e^{\sqrt{-1}\beta(t)}a,\
            -4\bigl(n\beta(t)+\vartheta(t)\bigr)\Bigr),
\qquad a\in S^{n-1} ,
\end{equation}
with the orientation of\/ $S^{n-1}$ of Example \ref{exa-round-slice} and the
$t$-direction last in the product orientation. The profile curve $re^{\sqrt{-1}\beta}$ has
derivative $e^{\sqrt{-1}(\beta+\vartheta)}$, so that it is parametrised by arc length, and
the height in \eqref{eqn-radial-model} satisfies
\begin{equation*}
\dot z=-4(n\dot\beta+\dot\vartheta)
=-\frac r2\sin\vartheta=-\frac12r^2\dot\beta .
\end{equation*}
Writing \eqref{eqn-LeM1} as $\eta=\frac12dz+\frac14\mathrm{Im}\sum_j\overline w_j\,dw_j$,
we find $\eta\bigl(F_*\frac{\p}{\p t}\bigr)=\frac12\dot z+\frac14 r\sin\vartheta=0$, which
is the contact condition in the $t$-direction, and $\eta$ vanishes on the spherical
directions because $a\cdot da=0$. Thus $F_{\mathrm{rad}}$ is Legendrian. In the standard
coordinates of\/ $\mathbb C^n$ the determinant of the coordinate frame
$\bigl(F_*a_1,\dots,F_*a_{n-1},F_*\frac{\p}{\p t}\bigr)$ is
$r^{n-1}e^{\sqrt{-1}(n\beta+\vartheta)}\det(a_1,\dots,a_{n-1},a)$, whose argument is
$n\beta+\vartheta$; the conjugation in the frame convention of \eqref{eqn-angledet}
therefore gives the global real branch
\begin{equation*}
\theta=-n\beta-\vartheta=\frac z4
\end{equation*}
of the Legendrian angle, and $F_{\mathrm{rad}}$ is a Legendrian self-shrinker by
Proposition \ref{pro-scalar}. Since $F_{\mathrm{rad}}^*\eta=0$, the induced metric is one
quarter of the pull-back of $g_{\mathrm{Euc}}$: the squared length of
$F_*\frac{\p}{\p t}$ is $\frac14$, that of a spherical direction is $\frac14r^2$ times
its squared length in $S^{n-1}$, and the cross terms vanish because $a\cdot da=0$. Hence
\begin{equation}\label{eqn-radial-metric}
g=\frac14\,dt^2+\frac14\,r(t)^2g_{S^{n-1}},\qquad f=\frac1{16}\,r(t)^2 .
\end{equation}

The first two equations in \eqref{eqn-radial-system} form an autonomous system in the
strip $r>0$, $0<\vartheta<\pi$, and this system has the first integral
\begin{equation*}
J(r,\vartheta)=r^n e^{-r^2/16}\sin\vartheta ,
\end{equation*}
as one checks by differentiating $\log J$. At $(R_0,\frac\pi2)$ the gradient of\/ $\log J$
vanishes and its Hessian is $\operatorname{diag}(-\frac14,-1)$, so that the level sets of\/
$J$ near this nondegenerate maximum are closed curves surrounding it. The only zero of the
system in the strip is $(R_0,\frac\pi2)$, so that each of these level curves is a periodic
orbit. It meets the line $\vartheta=\frac\pi2$ at two points with radii $r_-<R_0<r_+$,
since $r^ne^{-r^2/16}$ increases up to $R_0$ and decreases afterwards; hence
$f=\frac1{16}r^2$ takes values on both sides of\/ $\frac n2$ along it. As the orbits
contract to $(R_0,\frac\pi2)$, $r$ tends uniformly to $R_0$ and $\sup|f-\frac n2|$ tends
to zero.

Fix one such orbit. Along it $r$ and $\sin\vartheta$ are bounded below by positive
constants and $r$ is bounded above, so that the solution of \eqref{eqn-radial-system},
including $\beta$, is defined for all $t\in\mathbb R$; the metric
\eqref{eqn-radial-metric} is uniformly equivalent to a product metric, hence complete;
and $\dot z=-\frac r2\sin\vartheta$ is negative and bounded away from zero, so that
$z:\mathbb R\to\mathbb R$ is a proper bijection. Consequently the inverse image under
$F_{\mathrm{rad}}$ of a compact set has bounded $t$, hence is compact because $S^{n-1}$
is compact, and $F_{\mathrm{rad}}$ is proper. It is injective: two points with the same
image have the same height, hence the same $t$, and then the same $a$ by the nonvanishing
complex factor in \eqref{eqn-radial-model}. Thus $F_{\mathrm{rad}}$ is a complete properly
embedded Legendrian self-shrinker, and choosing the orbit sufficiently close to
$(R_0,\frac\pi2)$ gives the asserted $\varepsilon$-bound.

Any statement in which the threshold $8n$ of Theorem \ref{thm-cylinder} is increased on
the inner side, or decreased on the outer side, would admit one of these non-cylindrical
examples. Together with Proposition \ref{pro-sharp-curve} this proves sharpness in every
dimension.
\end{proof}

\end{document}